\documentclass[11 pt]{amsart}

\usepackage[latin1]{inputenc}
\usepackage{amsmath,amsthm,amssymb,amscd,hyperref}
\usepackage{graphicx}
\usepackage{color}

\newtheorem{thm}{Theorem}[section]
 \newtheorem{cor}[thm]{Corollary}
 \newtheorem{lem}[thm]{Lemma}
 \newtheorem{prop}[thm]{Proposition}

 \theoremstyle{definition}
 \newtheorem{df}[thm]{Definition}
 \theoremstyle{remark}
 \newtheorem{rem}[thm]{Remark}
 
 \numberwithin{equation}{section}

\def\be#1 {\begin{equation} \label{#1}}
\newcommand{\ee}{\end{equation}}

\def\sqw{\hbox{\rlap{\leavevmode\raise.3ex\hbox{$\sqcap$}}$%
\sqcup$}}
\def\findem{\ifmmode\sqw\else{\ifhmode\unskip\fi\nobreak\hfil
\penalty50\hskip1em\null\nobreak\hfil\sqw
\parfillskip=0pt\finalhyphendemerits=0\endgraf}\fi}

\newcommand{\R}{\mathbb R}

\newcommand\<{\langle}
\renewcommand\>{\rangle}

\usepackage{cite}

\newcommand{\mc}{\mathcal}
\newcommand{\mbf}{\mathbf}

\DeclareMathOperator{\supp}{supp}

\DeclareMathOperator{\dist}{dist}

\newcommand{\bw}{\mbf w}
\newcommand{\bg}{\mathbf g}
\newcommand{\ba}{\mathbf a}
\newcommand{\bff}{\mathbf f}
\newcommand{\cM}{\mathcal M}

\newcommand{\cK}{\mathcal K}
\newcommand{\X}{\mathcal H}
\newcommand{\cZ}{\mathcal Z}
\newcommand{\brho}{\pmb\rho}

\title[A quasilinear KdV equation with degenerate dispersion]{Existence and uniqueness of solutions for a quasilinear KdV equation with degenerate dispersion}

\author{Pierre Germain}
\email{pgermain@cims.nyu.edu}
\address{Courant Institute of Mathematical Sciences, New York University\\
251 Mercer Street, New York, NY 10012, USA}
\thanks{P.G. was supported by the NSF grant DMS-1501019.}
\author{Benjamin Harrop-Griffiths}
\email{benjamin.harrop-griffiths@cims.nyu.edu}
\address{Courant Institute of Mathematical Sciences, New York University\\
251 Mercer Street, New York, NY 10012, USA}
\thanks{B.H.-G. was supported by a Junior Fellow award from the Simons Foundation.}
\author{Jeremy L.~Marzuola}
\email{marzuola@math.unc.edu}
\address{Mathematics Department, University of North Carolina \\
Phillips Hall, Chapel Hill, NC 27599, USA}
\thanks{J.L.M. was supported in part by U.S. NSF Grants DMS--1312874 and DMS-1352353.}

\thanks{
The authors would like to thank Philip Rosenau for several informative discussions about degenerate dispersive equations and compactons, and for pointing out the elegant waiting time argument of~\cite{2017arXiv170903322Z}.
}

\begin{document}

\begin{abstract}
We consider a quasilinear KdV equation that admits compactly supported traveling wave solutions (compactons). This model is one of the most straightforward instances of degenerate dispersion, a phenomenon that appears in a variety of physical settings as diverse as sedimentation, magma dynamics and shallow water waves. We prove the existence and uniqueness of solutions with sufficiently smooth, spatially localized initial data.
\end{abstract}

\maketitle

\section{Introduction}\label{sect:Intro}

\subsection{The equation}

In this article we consider the existence and uniqueness of real-valued solutions \(u\colon \R_t\times \R_x\rightarrow\R\) of the quasilinear Korteweg-de Vries (KdV) equation,
\begin{equation}
\label{QdV}
\begin{cases}
u_t + \left( u(uu_x)_x + \mu u^3 \right)_x = 0,\vspace{0.1cm}\\
u(0) = u_0.
\end{cases}
\end{equation}
Here, the parameter $\mu$ can take the values $+1$ (focusing case), $-1$ (defocusing case), or $0$ (neutral case). This model appeared in~\cite{MR1376975} as a Hamiltonian variation of the degenerate dispersive models of Rosenau and Hyman~\cite{HR}.

The equation \eqref{QdV} may be formally derived from the Hamiltonian,
\begin{equation}\label{Hamiltonian}
H(u) := \frac{1}{2} \int |uu_x|^2\,dx - \frac \mu 4 \int u^4\,dx,
\end{equation}
given the symplectic form \(\omega(u,v) = \int u\cdot\partial_x^{-1}v\,dx\). In addition to the usual translation and reflection symmetries, the equation \eqref{QdV} enjoys the scaling symmetry,
\begin{equation}\label{Scaling}
u(t,x) \mapsto \sqrt{\lambda}u(\lambda t, x),\qquad \lambda>0,
\end{equation}
which makes it $L^2$-subcritical.

Formally, the flow of the equation~\eqref{QdV} conserves, besides the Hamiltonian, the mass $M$ and the momentum $J$ defined by
\[
M(u) := \int u^2 \,dx, \qquad J(u) := \int u \,dx.
\]
Finally, denoting $u_+ = \max(0,u)$, the positive momentum
$$
J_+(u) := \int u_+ \,dx,
$$
is also conserved for smooth solutions: indeed,
$$
\frac{d}{dt} \int u_+ \,dx = - \int \frac{u_+}{u} \left( u(uu_x)_x + \mu u^3 \right)_x \,dx 
= \int \delta(u)\ u_x \left( u(uu_x)_x + \mu u^3 \right)\,dx = 0.
$$

\subsection{Degenerate dispersive equations} 
The equation~\eqref{QdV} is one of the simplest instances of degenerate dispersive equations: the dispersive term is superlinear, so the dispersive effect degenerates as $u \to 0$.

Degenerate dispersive equations occur in the description of a number of physical phenomena. To name a few: sedimentation~\cite{MR1021642,MR2772627}; dynamics of magma~\cite{MR2285103,MR2434917}; granular media~\cite{PORTER2009666, Nesterenko}; shallow water waves with the Camassa-Holm equation~\cite{MR1234453,MR2278406} and Green-Naghdi equations~\cite{2017arXiv171003651L}; liquid crystals with the Hunter-Saxton equation~\cite{MR1135995}; elasticity~\cite{MR1811323}; nonlinear chains dynamics~\cite{MR1928031,DMR}; cosmology~\cite{MR2728788}. More recently, degenerate dispersive equations were found to describe waves propagating on interfaces~\cite{MR2599457,MR2857003} and even to provide a model for weak turbulence~\cite{MR2651381,MR3171091}.

Similar types of degenerate behavior occur in other PDE contexts: gradient flows such as the porous medium equation or the parabolic p-Laplacian flow (see for example the monographs~\cite{MR2286292,MR1230384}); higher order diffusion such as the thin film equation ~\cite{MR2839301,MR3375546,MR3375536,MR3197240,MR3544328}; weakly hyperbolic equations~\cite{MR1378250,MR1668860}, in particular the compressible Euler equations near vacuum~\cite{MR2547977,MR2779087}. Indeed, many of the techniques in this paper owe inspiration to previous work on degenerate parabolic and hyperbolic equations.

\subsection{Compactons} A feature of many degenerate dispersive equations is that they support \textit{compactons}: traveling waves with compact support. This was first emphasized by Rosenau and Hyman~\cite{HR, MR1294558}, who introduced the model \(K(m,n)\) equations,
\begin{equation}\label{KMN}
u_t + (u^m)_x + (u^n)_{xxx} = 0.
\end{equation}
Subsequently, numerous classes of degenerate dispersive equations exhibiting an array of remarkable traveling wave solutions have been introduced and studied. We refer the reader to the forthcoming review article~\cite{RosenauReview} of Rosenau and Zilburg and the papers~\cite{ROSENAU200644,MR2159688,MR2133462,MR2601818,MR3628993,HR,MR1294558} for a more detailed history of these problems and some recent results.

In a recent article~\cite{GHGM} the authors considered the equation \eqref{QdV} in the focusing case $\mu > 0$, and established the variational properties of several families of traveling wave solutions. They actually worked in a more general framework, where $u^4$ is replaced by $|u|^p$, with $p \geq 2$, in the Hamiltonian \(H(u)\); for simplicity we restrict our attention in this paper to the case $p =4$. The explicit compacton solutions are then given by
\[
u(t,x) = \Phi_{B,c}(x - ct),
\]
where either \(B = 0\), \(c>0\) or \(B>0\), \(c\in \R\) and we define
\begin{equation}\label{PhiBc}
\Phi_{B,c}(x) := \sqrt{c + \sqrt{4B + c^2}\cos(\sqrt 2x)},\qquad x\in \left(-x_{B,c},x_{B,c}\right),
\end{equation}
where \(x_{B,c}>0\) is the smallest positive solution to \(\cos(\sqrt 2x) = -\frac c{\sqrt{4B + c^2}}\). We note that in contrast to the usual KdV equation the compactons may travel in either the positive or negative direction (or even remain stationary). Further, the \(B = 0\) compactons are the minimizers of the Hamiltonian for fixed mass (see~\cite[Theorem~1.2]{GHGM}).

\subsection{Degenerate initial data}
Local well-posedness for non-degenerate initial data (say, perturbations of a constant or of  strictly positive traveling wave solutions) may be obtained from the result of Akhunov~\cite{MR2982797}, building on the work of Kenig-Ponce-Vega~\cite{MR2096797,MR1195480,MR1321214} (see also \cite{MR2955206,MR3263550,MR3280026,MR3417686,MR3179690,MR1468372,MR2446185,MR1448822}). Thus, we now restrict our attention to degenerate initial data.

One motivation for considering degenerate initial data is the stability of compactons: we saw that they are variationally stable if $B=0$... but it seems to be very difficult to construct solutions to the equation~\eqref{QdV} (in any sense) for perturbations (in any topology) of the compactons. In other words, this leads to the question:

\medskip
\noindent\textbf{Question:} Do there exist finite mass / energy solutions to \eqref{QdV} for initial data in an open neighborhood of the compacton solutions (in a suitable topology)?
\medskip

The main goal of this article is to take a first step towards answering this question, by proving local existence and uniqueness of solutions to the equation~\eqref{QdV} for suitable initial data, although we note that our initial data does not include the compacton solutions themselves.

In a recent article~\cite{2017arXiv170903322Z} Zilburg and Rosenau show that classical solutions to \eqref{QdV} in the focusing case must lose regularity in finite time, and that sufficiently smooth solutions obey a ``waiting time'' effect analogously to solutions of degenerate parabolic problems outlined in~\cite{MR1230384}. In Section~\ref{sect:Virial} we briefly sketch their argument and show that it may be adapted to the cases \(\mu = 0,-1\). As a consequence, the solutions constructed in the present article, which have fixed support, will either develop a singularity, or start moving in finite time.

The existence of global weak solutions for a degenerate KdV equation similar to~\eqref{QdV} admitting compactons was previously considered by Ambrose and Wright~\cite{MR3091772}.  The same authors considered the existence of classical solutions to another related model~\cite{MR2586373}. However, previous existence proofs have relied on the presence of higher order conservation laws for solutions, giving a priori control of higher order Sobolev norms. In this article we use the toy model \eqref{QdV} to develop a rather more robust proof of the existence of solutions. Indeed, our proof does not explicitly use the Hamiltonian structure of \eqref{QdV}, but rather the existence of a hydrodynamic formulation (see~\eqref{transport}), and hence we expect it may be applied to a much broader class of degenerate dispersive equations. In particular, we expect that our argument can be applied to obtain existence and uniqueness of solutions to the \(K(m,n)\) equations (defined as in \eqref{KMN}) whenever \(m\geq 1\) and \(n\geq 3\).

\subsection{Endpoint decay rates}
\label{subsectionedr}
We will subsequently assume that the initial data \(u_0\) for \eqref{QdV} is the positive square root of a continuous non-negative function \(\rho = u_0^2\) with simply connected set of positivity
\[
I := \left\{x\in \R:\rho(x)>0\right\}.
\]
In this subsection, it is a bounded open interval, \(I = (x_-,x_+)\).

In order to understand the effect of the endpoint decay on the solution, we consider the leading order part of the linearization of \eqref{QdV} about the initial data,
\begin{equation}\label{QuantumLinear}
u_t + \rho u_{xxx} = 0.
\end{equation}
In the semi-classical regime, if $u_0$ is initially localized in phase space around $(x_0,\xi_0)$, the solution \(u\) to the equation \eqref{QuantumLinear} will be localized on the bicharacteristics of the symbol \(a(x,\xi) = - \rho(x)\xi^3\), given by the classical Hamiltonian flow
\begin{equation}\label{HamiltonianODE}
\begin{cases}
\dot x = a_\xi(x,\xi) = - 3\rho(x)\xi^2,\vspace{0.1cm}\\
\dot\xi = - a_x(x,\xi) = \rho_x(x)\xi^3.
\end{cases}
\end{equation}
Suppose that \(I = \R_+\) and \(\rho(x) = x^k\) for \(0<x\ll1\). We may then explicitly solve the equation \eqref{HamiltonianODE} with initial data \((x_0,\xi_0)\) for some \(0<x_0\ll1\) to obtain,
\begin{align*}
x(t) &= \begin{cases}x_0\left(1 + (k - 3)x_0^{k - 1}\xi_0^2t\right)^{\frac3{3 - k}},&\qquad k\neq 3,\vspace{0.1cm}\\x_0e^{-3x_0^2\xi_0^2t},&\qquad k = 3,\end{cases}\\
\xi(t) &= \begin{cases}\xi_0\left(1 + (k - 3)x_0^{k - 1}\xi_0^2t\right)^{\frac k{k - 3}},&\qquad k\neq 3,\vspace{0.1cm}\\\xi_0e^{3x_0^2\xi_0^2t},&\qquad k = 3.\end{cases}
\end{align*}
In particular, whenever \(k<3\) the frequency will blow up in finite time, whereas when \(k\geq 3\) the frequency blows up in infinite time.

These heuristics suggest that solutions to \eqref{QuantumLinear} may form singularities instantaneously whenever \(k< 3\), whereas one can hope for well-defined solutions on sufficiently short time intervals whenever \(k\geq 3\). As a consequence we make the following definition:
\begin{df} After translation, assume that \(0\in I\). We say that \(u_0\) has \emph{supercritical left endpoint decay} if \(\rho = u_0^2\) satisfies
\[
\int_{x_-}^0 \frac1{\rho(s)^{\frac13}}\,ds < \infty,
\]
Similarly, we say that \(\rho\) has \emph{supercritical right endpoint decay} if
\[
\int_0^{x_+} \frac1{\rho(s)^{\frac13}}\,ds < \infty.
\]
\end{df}
Unfortunately our existence result will not hold for all data without supercritical left endpoint decay, but rather initial conditions for which the frequency grows at a sub-exponential rate. As a consequence we make a further definition:
\begin{df}\label{df:Subcrit} After translation, assume that \(0\in I\). We say that \(u_0\) has a \emph{subcritical left endpoint decay} if \(x_- = -\infty\) or \(\rho = u_0^2\) satisfies
\begin{equation}\label{LEP}
\rho(x) = o\left(\dist(x,x_-)^3\right),\qquad x\downarrow x_-.
\end{equation}
Similarly, we say that \(u_0\) has a \emph{subcritical right endpoint decay} if \(x_+ = \infty\) or \(\rho = u_0^2\) satisfies
\begin{equation}\label{REP}
\rho(x) = o\left(\dist(x,x_+)^3\right),\qquad x\uparrow x_+.
\end{equation}
If \(\rho\) has neither subcritical nor supercritical left (respectively right) endpoint decay we say it has \emph{critical left (respectively right) endpoint decay}.
\end{df}

We note that, provided \(\rho\) is sufficiently smooth, the bicharacteristics will leave a small neighborhood of the right endpoint eventually, leading to a smoothing effect near \(x_+\). Consequently, we do not expect the right endpoint decay to significantly affect the existence of solutions, only the left endpoint decay. However, in this article we restrict our attention to the case of subcritical right endpoint decay. The more involved case of critical or supercritical right endpoint decay will be addressed in a future article.

\subsection{Hydrodynamic solutions and the main result}
We observe that the equation \eqref{QdV} may be written in the hydrodynamic form
\begin{equation}
\label{transport}
u_t + (bu)_x = 0, \quad \mbox{where} \quad b = \frac{1}{2} (u^2)_{xx} + \mu u^2.
\end{equation}
The equation \eqref{transport} makes sense whenever \(u\in C^1([0,T]\times\R)\) and \(b\in C([0,T];C^1(\R))\). This motivates the following definition:
\begin{df}\label{def:hydro}
Given \(T>0\) we say that a non-negative function \(u\in C^1([0,T]\times\R)\) is a \emph{hydrodynamic solution} of \eqref{QdV} if \(u^2\in C([0,T];C^3(\R))\) and \(u\) satisfies the equation \eqref{transport} for all \((t,x)\in[0,T]\times \R\).
\end{df}
Evidently, classical solutions to \eqref{QdV}, i.e. \(u\in C([0,T];C^3(\R))\cap C^1([0,T];C(\R))\), are hydrodynamic solutions. However, if either endpoint is finite this definition allows for the case that \(u(t,x)= o(\dist(x,x_\pm)^{\frac32})\) as \(x\rightarrow x_\pm\), which is subcritical endpoint decay in the sense of Definition~\ref{df:Subcrit}, but not a classical solution. We note that this definition is not restricted to solutions that vanish at infinity and hence includes non-degenerate solutions.


To further motivate this definition we have the following uniqueness result, the proof of which is delayed to Section~\ref{sect:Uniqueness}:
\begin{thm}\label{thm:Unique}
Given non-negative initial data \(u_0\in C_b^1(\R)\) so that \(u_0^2\in C^3_b(\R)\) and \(u_0^{\frac43}\in C^2_b(\R)\) there exists at most one hydrodynamic solution of \eqref{QdV} so that \(u^2\in C([0,T];C^3_b(\R))\) and \(u^{\frac43}\in C([0,T];C^2_b(\R))\).
\end{thm}

Here we write \(C_b^k = C^k\cap W^{k,\infty}\). The restriction that \(u(t)^{\frac43}\in C^2_b\) is required to rule out the possibility that \(u^2\) vanishes quadratically at an isolated zero and it seems reasonable to expect this may be replaced by assuming that \(u_0^2\in C^3_b(\R)\) has simply connected set of positivity.

Our main result is then (roughly) the following:
\begin{thm}[Rough statement]\label{thm:MainRough}
Let $u_0$ be sufficiently smooth and with simply connected set of positivity. Assume further that
\begin{itemize}
\item either $u_0$ is compactly supported with subcritical left and right endpoint decay
\item or \(u_0\) is supported on \(\R\) and asymptotically approaches a bounded non-degenerate traveling wave or zero.
\end{itemize}
Then there exists a time \(T>0\) and a unique hydrodynamic solution of the equation \eqref{QdV} on the time interval \([0,T]\).
\end{thm}

A rigorous statement of Theorem~\ref{thm:MainRough} is given in Theorem~\ref{thm:MainFull}.

\begin{rem}
The regularity and localization assumptions on the initial data are roughly:
\begin{enumerate}
\item There exists some integer \(K_0\geq 0\) so that \(\left(\int_0^x\rho(s)^{-\frac13}\,ds\right)^{\frac 65 K_0}\rho(x)\gtrsim 1\) (see \eqref{SlowDecayInit})
\item For integers \(1\leq n\leq 2K_0 + 4\) we have \(\partial_x^n\left(\rho^{\frac n3 + \frac16}\right)\in L^2\) (see \eqref{InitBound})
\end{enumerate}
\end{rem}

\begin{rem}
For compactly supported data, our result is essentially optimal as far  as endpoint decay rates are concerned: we can handle all smooth initial data that satisfies \(\rho\sim \dist(x,x_\pm)^{\alpha_\pm}\) as \(x\rightarrow x_\pm\) for \(\alpha_\pm>3\) (see Section~\ref{sect:Examples}). This is optimal in two respects:
\begin{itemize}
\item In the light of the bicharacteristic computation done in~\eqref{subsectionedr}.
\item In the hydrodynamic formulation, \(\alpha_\pm>3\) corresponds to requiring $b \in \mathcal{C}^1$. This is essentially optimal if one wants to define characteristics by the Picard-Lindel\"of theorem.
\end{itemize}
\end{rem}

\begin{rem}
It is possible to obtain a quantitative lower bound for the lifespan of existence from our result, although as it is likely far from optimal we do not attempt to track it carefully. However, it is clear from the proof that the lifespan depends not only on the size of the initial data, but also on the rate of decay of the initial data at the endpoints of its support and on the smallest local minimum of the initial data on the set of positivity \(I\).
\end{rem}

\subsection{Strategy of the proof} We now outline the strategy of the proof. The first difficulty is to give the equation an appropriate form to derive energy estimates. This is done in several steps:

\bigskip

\noindent \underline{Lagrangian formulation.} A key difficulty of working in the original frame is that the degeneracy at the endpoints will be time-dependent. In order to remove this time-dependence we switch to a moving frame, an approach that is common in degenerate hyperbolic and parabolic equations (see for example~\cite{MR1378250,MR1668860,MR3197240,MR3544328,KochHT}).
Recalling the hydrodynamic formulation \eqref{transport}, we let $X$ be the Lagrangian map associated to the vector field $b$; in other words
\begin{equation}\label{VFFlow}
\begin{cases}
X_t(t,x) = b(t,X(t,x))\vspace{0.1cm}\\
X(0,x) = x.
\end{cases}
\end{equation}
Letting \(Z(t,x) = \dfrac1{X_x(t,x)} - 1\) and $\rho = (u_0)^2$, the Cauchy problem for $u$ is equivalent to
\begin{equation}\label{Zeqn1}
\begin{cases}
Z_t + (1 + Z)^2\left(\frac12(1 + Z)\left((1 + Z)\left((1 + Z)^2\rho\right)_x\right)_x + \mu (1 + Z)^2\rho\right)_x = 0\vspace{0.1cm}\\
Z(0,x) = 0.
\end{cases}
\end{equation}

\bigskip

\noindent \underline{Change of independent coordinates and the Mizohata condition.}
The linearized problem for $Z$ about \(Z = 0\) reads $Z_t + \rho Z_{xxx} + \dots = 0$. In order to make the leading order coefficient constant, we set
$$
y = \int_0^x \frac1{\rho(s)^{\frac13}} \, ds,
$$
so that the linearized problem around $Z=0$ becomes
\begin{equation}
\label{linearizedZ}
Z_t + Z_{yyy} + \frac{5}{2} \frac{\rho_y}{\rho} Z_{yy} + \dots = 0
\end{equation}
The top order term has a constant coefficient, which greatly simplifies estimates. However, a new problem arises since this linearized problem violates the Mizohata condition: recall that a necessary condition for (forwards in time) local well-posedness in Sobolev spaces of the equation
\begin{equation}\label{NaiveToy}
w_t + w_{yyy} + a(y)w_{yy} = 0,
\end{equation}
on \(\R\), where \(a\) is assumed to be smooth and bounded, is the Mizohata condition~\cite{MR2967120,MR3179690,MR3266990,MR860041}
\begin{equation}\label{Mizohata}
\sup\limits_{y_1\leq y_2}\int_{y_1}^{y_2}a(s)\,ds <\infty.
\end{equation}
Thus, when proving local well-posedness for non-degenerate quasilinear KdV equations one typically assumes additional \(L^1\)-type integrability conditions for the initial data to ensure the condition~\eqref{Mizohata} is satisfied. Indeed, one may take advantage of the failure of~\eqref{Mizohata} to obtain ill-posedness in Sobolev spaces for quasilinear problems~\cite{MR2967120,MR2446185}. We remark that for non-degenerate initial data the Hamiltonian structure of \eqref{QdV} may be used to remove integrability conditions and prove local well-posedness in Sobolev spaces.

In order to circumvent this difficulty and obtain well-posedness for \eqref{NaiveToy} we must consequently work in a different topology. This relies on two key observations about linear KdV-type equations of the form \eqref{NaiveToy}: first, introducing the weight \(\Phi = e^{\frac13 A}\), where \(A\) is an antiderivative of \(a\), we may obtain energy estimates for \eqref{NaiveToy} in the weighted space \(L^2(\Phi^2\,dx)\). Indeed, integrating by parts yields,
\[
\frac d{dt}\|\Phi w\|_{L^2}^2 \lesssim \|a\|_{W^{3,\infty}}\|\Phi w\|_{L^2}^2.
\]
Second, for sufficiently smooth initial data, polynomial weights are propagated by the linear KdV flow on \(O(1)\) timescales. This is most readily seen from the identity \([\partial_t + \partial_y^3,y - 3t\partial_y^2] = 0\), and leads to the definition of the weighted Sobolev spaces \(\X^{N,K}\) in Section~\ref{sect:FunctionSpaces}.

Returning to the linearization~\eqref{linearizedZ} of the equation for \(Z\), we see that we should take \(\Phi = \rho^{\frac56}\). Due to the subcritical endpoint decay assumptions, $\rho(y)$ decays polynomially (as will be illustrated on several examples below); this implies that the Mizohata condition is barely violated and we can use the fact the the linear KdV equation propagates polynomial weights on \(O(1)\) timescales to prove the existence of solutions of the equation \eqref{Zeqn1} in weighted Sobolev spaces of the type $L^2(\rho^{\frac53}\,dx)$. However, it will be more convenient to perform one last change of coordinates...

\bigskip
\noindent \underline{Change of dependent coordinates.} Motivated by the previous discussion, we set $W = \rho^{5/6} Z$; it satisfies the equation
\begin{equation}
\label{equationW}
W_t + (1+g) W_{yyy} + \frac{7}{5} g_y W_{yy} + N(y,W,W_y) + \rho^{5/6} F = 0,
\end{equation}
where $N$ is a polynomial in $W$ and $W_y$, $F$ is a function of $\rho$ and its derivatives, and
$$
g = (1+ \rho^{-5/6} W)^5 - 1.
$$
This is the form of the equation which we will use to perform estimates.

Notice that the linearization of the equation around 0 now reads
$$
W_t + W_{yyy} + \{ \mbox{terms of order less than $1$} \} = 0;
$$
in particular the quadratic term $W_{yy}$ term has disappeared.

\bigskip
\noindent
\underline{Construction of solutions.} The equation~\eqref{equationW} is the one which we use to construct solutions. Our scheme is the following
\begin{itemize}
\item Regularization of the equation is achieved by adding a term $-\nu W_{yyyy}$ on the right-hand side of~\eqref{equationW}. This allows construction of local solutions over a time span $[0,T(\nu)]$.
\item Energy estimates in weighted Sobolev spaces allow to obtain a uniform time of existence (in $\nu$) as well as uniform bounds on the corresponding solutions. These energy estimates are delicate, and rely crucially on the structure of~\eqref{equationW}.
\item Finally, a simple compactness argument allows us to pass to the limit $\nu \to 0$, first finding a convergent sequence of solutions, and then passing to the limit in the equation.
\end{itemize}


\section{Some technical estimates}
\subsection{Function spaces}\label{sect:FunctionSpaces}
We will seek solutions to the equation \eqref{equationW} in the weighted Sobolev spaces \(\X^{N,K}\subset L^2\), defined to be the completion of \(C^\infty_0(\R)\) under the norm,
\begin{equation}\label{FunctionSpace}
\| f \|_{\X^{N,K}} = \sum\limits_{k = 0}^K\sum\limits_{n = 0}^{2(K - k) + N}\|\< y\>^k\partial_y^n f\|_{L^2}.
\end{equation}

Before recording some of the basic properties of the $\X^{N,K}$ space, let us explain in a few sentences why this space is adapted to the (flat) Airy equation. Arguing heuristically, consider data of $L^2$ mass one, which are localized in phase-space around $(x_0,\xi_0)$; this gives a norm in $\X^{N,K}$ of order $A \sim \sum_{k=0}^K x_0^k \xi_0^{2(K-k) + N}$. At time $t=1$, the solution of the Airy equation should be localized in phase space around $(x_0-3 \xi_0^2,\xi_0)$, giving a norm of order $B \sim \sum_{k=0}^K (x_0 - 3 \xi_0^2)^k \xi_0^{2(K-k) + N}$. Since $B \lesssim A$, it should be expected that the Airy equation is locally well-posed in $\X^{N,K}$.

Turning to the properties of $\X^{N,K}$, we remark first that 
\[
\|\partial_y^nf\|_{\X^{N,K}}\lesssim \|f\|_{\X^{N + n,K}},
\]
and that we have the interpolation estimate,
\begin{equation}\label{Interpolate}
\| f \|_{\X^{4,K}}^2\lesssim \| f \|_{\X^{1,K}}\| f \|_{\X^{7, K}}.
\end{equation}
Further, if \(P_{\leq j}\) is the usual Littlewood-Paley projection to frequencies \(\lesssim 2^j\) and \(P_{>j} = 1 - P_{\leq j}\) we have the estimate for $j\geq0$
\begin{equation}\label{Bernie}
\| P_{\leq j}f \|_{\X^{N + n,K}}\lesssim 2^{nj}\| f\|_{\X^{N,K}},\qquad \| P_{>j} f\|_{\X^{N,K}}\lesssim 2^{-nj}\| f \|_{\X^{N + n,K}}.
\end{equation}

We will construct solutions using a parabolic regularization given by the semigroup \(e^{-\nu t\partial_y^4}\). This motivates defining the subspace \(\cZ^{N,K}\subset C([0,T];\X^{N,K})\) with norm
\begin{equation}\label{SNorm}
\| f \|_{\cZ^{N,K}} = \sum\limits_{ n = 0}^3 \|(\nu t)^{\frac n4} f \|_{L^\infty([0,T];\X^{N + n,K})}
\end{equation}
We then have the following lemma:
\begin{lem}
Let \(N,K\geq 0\), \(0<\nu \ll1\) and \(T\lesssim \nu^{-1}\). Then for all \(G\in C^\infty_0([0,T]\times\mathbb{R})\) we have the estimates
\begin{align}
\|e^{-\nu t\partial_x^4}f\|_{\cZ^{N,K}}&\lesssim \|f\|_{\X^{N,K}}\label{semigroupInit}
\\\label{ParabolicReg}
\left\|\int_0^t e^{-\nu(t - s)\partial_y^4}G(s)\,ds\right\|_{\cZ^{N,K}}&\lesssim \nu^{-\frac34}T^{\frac14} \|(\nu t)^{\frac34}G\|_{L^\infty([0,T];\X^{N,K})},
\end{align}
where the constant is independent of \(\nu\).
\end{lem}
\begin{proof} Starting with the case $K=0$, denote the Fourier transform of $e^{-\xi^4}$ by $\psi$. The kernel of $\partial_y^n e^{-\nu(t-s) \partial_y^4}$ is then given by
$$
\frac{1}{(\nu(t-s))^{\frac{n+1}{4}}} \psi^{(n)} \left( \frac{x}{(\nu(t-s))^{1/4}} \right),
$$
which has $L^1$ norm $\lesssim \frac{1}{(\nu(t-s))^{n/4}}$. Therefore, applying Minkowski's inequality we obtain
$$
(\nu t)^{n/4}\|\partial_y^ne^{-\nu t\partial_x^4}f\|_{H^N}\lesssim \|f\|_{H^N},
$$
and similarly,
$$
(\nu t)^{n/4} \left\| \int_0^t \partial_y^n e^{-\nu(t-s) \partial_y^4} G(s)\,ds \right\|_{H^N} \lesssim \int_0^t \frac{1}{(1-\frac{s}{t})^{n/4}} \|G(s) \|_{H^N} \,ds \lesssim t^{\frac14}\| t^{\frac34} G \|_{L^\infty H^N},
$$
where we used the fact that for \(0\leq n\leq 3\),
\[
\int_0^1 \frac1{(1 - s)^{n/4}s^{3/4}}\,ds\lesssim 1.
\]
This gives the desired result for $K=0$.

Turning to the case $K \geq 1$, observe that
$$
[ e^{-\nu t \partial_y^4},y] = - 4 \nu t \partial_y^3 e^{-\nu t \partial_y^4}.
$$
Therefore, $[y^K \partial_y^n,e^{-\nu t \partial_y^4}]$ is a linear combination of terms of the type
$$
\partial_y^n (\nu t \partial_y^3)^\ell e^{-\nu t \partial_y^4} y^m, \qquad \mbox{with}\;\; \ell + m \leq K,
$$
where the kernel of \(\partial_y^n (\nu t \partial_y^3)^\ell e^{-\nu t \partial_y^4}\) is given by
$$
(\nu t)^{\frac{\ell-n-1}{4}} \psi^{(3\ell+n)} \left( \frac{x}{(\nu t)^{1/4}} \right),
$$
which has $L^1$ norm $\lesssim (\nu t)^{\frac{\ell-n}{4}} \lesssim (\nu t)^{-n/4}$ if $\nu t \lesssim 1$. Arguing as before then gives the desired inequality.
\end{proof}

\subsection{Pointwise bounds}
In order to control the pointwise behavior of solutions we first recall the usual \(1d\) Sobolev estimate,
\begin{equation}\label{BasicSobolev}
\|f\|_{L^\infty}\lesssim \|f\|_{H^1}.
\end{equation}
Applying this estimate to \(\<y\>^k \partial_y^n f\) we obtain the following lemma:
\begin{lem}
If \(0\leq k\leq K\) and \(f\in C^\infty_0(\R)\) we have the estimate,
\begin{equation}\label{Sobolev}
\sum\limits_{n = 0}^{2(K - k) + N - 1}\|\<y\>^k\partial_y^nf\|_{L^\infty}\lesssim \|f\|_{\X^{N,K}}.
\end{equation}
\end{lem}

\begin{rem}\label{rem:Compact1}
We recall that the embedding \eqref{BasicSobolev} is locally compact. As a consequence, the embedding \eqref{Sobolev} is also locally compact.
\end{rem}

\begin{rem}\label{rem:OddSobolev}
In our application of the pointwise estimate \eqref{Sobolev} to control products of functions in \(\X^{N,K}\) we will require a slight refinement when \(N\) is even and \(n\) is odd. Suppose that $0 \leq k \leq K -1$, that $1 \leq n \leq 2(K-k) + N$, and that \(N\) is even while $n$ is odd. Then, for any \(f\in C^\infty_0(\R)\) we may integrate by parts to obtain
\begin{align*}
\|\< y\>^{k + \frac12} \partial_y^nf\|_{L^2}^2 & = - \int \<y\>^{2k + 1} \partial_y^{n + 1}f\cdot \partial_y^{n - 1}f\,dy - (2k+1)\int y\<y\>^{2k - 1}\partial_y^nf\cdot\partial_y^{n - 1}f\,dy \\
& \leq \| \langle y \rangle^{k+1} \partial_y^{n-1} f \|_{L^2} \| \langle y \rangle^{k} \partial_y^{n+1} f \|_{L^2} + \| \langle y \rangle^{k+1} \partial_y^{n-1} f \|_{L^2} \| \langle y \rangle^{k-1} \partial_y^{n} f \|_{L^2} \\
& \lesssim \|f\|_{\X^{N,K}}^2,
\end{align*}
where we have used the fact that as \(n\) is odd while \(N\) is even we have \(n - 1\leq 2(K - (k + 1)) + N\) and \(n + 1\leq 2(K - k) + N\). Applying the usual \(1d\) Sobolev estimate \eqref{BasicSobolev} then yields the slight refinement of the pointwise estimate \eqref{Sobolev},
\[
\|\<y\>^{k+\frac12}\partial_y^{n-1} f\|_{L^\infty}\lesssim \|f\|_{\X^{N,K}}.
\]
\end{rem}

\subsection{Product laws}
Given sufficiently smooth functions \(f_1,\dots,f_M\) we define the multilinear operator
\[
L_n[y,f_1,\dots,f_M] = \sum\limits_{\substack{|\alpha|\leq n+3\\\max\alpha_m\leq n}}\left( C_\alpha(y)\delta^{-(M - 1)}\<y\>^{K(M - 1)}\prod\limits_{m = 1}^M \partial_y^{\alpha_m}f_m(y)\right),
\]
where \(\delta>0\), \(K\geq 0\) and we assume that the coefficients \(C_\alpha\) are smooth, uniformly bounded functions. This type of multilinear expression will appear in the perturbative terms of the equation for \(W\) considered in Section~\ref{sect:Main}. We will also use this as a notation, writing
\[
G = L_n[y,f_1,\dots,f_M],
\]
if a multilinear operator \(G\) may be written in this form.

We then have the following estimate for multilinear operators of this form:
\begin{lem}\label{lem:Multilinear}
Let \(0\leq k\leq K\) and \(0\leq n\leq 2(K - k) + 4\). Then if \(f_1,\dots,f_M\in C^\infty_0(\R)\) and \(C_\alpha\in C^\infty(\R)\) we have the estimate
\begin{equation}\label{MultilinearBound}
\|\<y\>^kL_n[y,f_1,\dots,f_M]\|_{L^2}\lesssim \delta^{-(M-1)}\left(\sum_\alpha \|C_\alpha\|_{L^\infty}\right)\left(\prod\limits_{m=1}^M \|f_m\|_{\X^{4,K}}\right).
\end{equation}
\end{lem}

\begin{proof} \underline{First reduction.} After reordering the indices, we can assume that $\alpha_1 \leq \alpha_2 \leq \dots \leq \alpha_M$.

Let us see quickly why the result is easy if the $\alpha_i$ are sufficiently small: first, if $\alpha_M \leq 3$, the desired bound easily holds, since $k \leq K$. Second, assume that $\alpha_1 \dots \alpha_{m_0}$ are all $\leq 2$, with $1 \leq m_0 \leq K-1$. Then one can estimate $\langle y \rangle^K \partial_y^{\alpha_m} f_m$ in $L^\infty$, for $m \leq m_0$, and matters reduce to proving the desired result for
$$
L_n'[y,f_{m_0+1},\dots,f_M] = \sum\limits_{\substack{|\alpha|\leq n+3\\\max\alpha_m\leq n}}\left( C_\alpha(y)\delta^{-(M - m_0 - 1)}\<y\>^{K(M - m_0 - 1)}\prod\limits_{m = m_0+1}^M \partial_y^{\alpha_m}f_m(y)\right).
$$
In other words, we can assume in the following that $\alpha_m \geq 3$ for $m \leq M-1$ and $\alpha_M \geq 4$.

Also notice that the case $k=0$ is easily dealt with; therefore, we shall assume in the following that $k \geq 1$.

\bigskip

\noindent
\underline{The case $M=1$.} It is immediate.

\bigskip

\noindent \underline{The case $M \geq 3$.} 
Our aim is to bound in $L^2$ 
$$
\langle y \rangle^{k + K (M-1)} \prod_{m=1}^M \partial^{\alpha_m}_y f_m,
$$
where $\alpha_1 \leq \alpha_2 \leq \dots \leq \alpha_M$, under the assumption that $f_m \in \X^{4,K}$ for all $m$.

The idea is to estimate the function carrying the most derivatives, namely $f_M$, in $L^2$, and all the others in $L^\infty$. Observe that
\begin{itemize}
\item On the one hand, $\| \langle y \rangle^{\beta_M} \partial_y^{\alpha_M} f_M \|_{L^2} \lesssim \| f_M \|_{\X^{4,K}}$ provided $\alpha_M \leq 2K - 2\beta_M +4$ and $\beta_M \leq K$; and the latter condition follows from the former since $\alpha_M \geq 4$.

\item On the other hand, if $m \leq M-1$, $\| \langle y \rangle^{\beta_m} \partial_y^{\alpha_m} f_m \|_{L^\infty} \lesssim \| f_M \|_{\X^{N,K}}$ provided $\alpha_m \leq 2K - 2\beta_m + 3$ and $\beta_m \leq K$; and the latter condition follows from the former since $\alpha_m \geq 3$.
\end{itemize}
Since $\beta_M$ and $\beta_m$ must be integers, the best choice possible is
$$
\beta_M = \left \lfloor K+2-\frac{\alpha_M}{2} \right \rfloor \qquad \mbox{and} \qquad \beta_m = \left \lfloor K+\frac{3}{2}-\frac{\alpha_m}{2} \right \rfloor \quad \mbox{for $1\leq m \leq M-1$}
$$
(notice that $\beta_m \geq 0$ for all $m$ since $k \geq 1$).
In order for the desired bound to hold, we need $\sum_{m=1}^M \beta_m \geq k + K (M-1)$, which follows from
$$
K+2-\frac{\alpha_M}{2} + \sum_{m=1}^{M-1} \left( K+\frac{3}{2}-\frac{\alpha_m}{2}\right) - \frac{M}{2} \geq k + K (M-1),
$$
where the summand $- \frac{M}{2}$ on the left-hand side comes from the rounding errors. Since $\sum \alpha_m \leq n+3$, the above inequality holds if
$$
n \leq 2K - 2k + 2M -2.
$$
This inequality is satisfied, under the assumptions of the lemma, if $M \geq 3$.

\bigskip 

\noindent \underline{The case $M=2$.} The above argument suffices if $\alpha_1 + \alpha_2 \leq 2(K-k) + 5$. Further, in the case that \(\alpha_1 + \alpha_2 = 2(K - k) + 6\) the \(\alpha_j\) have the same parity, so taking \(\beta_1,\beta_2\) as above we see that \(\beta_1 + \beta_2 = k + K\). In the remaining case \(\alpha_1 + \alpha_2 = n + 3 = 2(K - k) + 7\), we take
\[
\beta_1 = K + \frac32 - \frac{\alpha_1}2,\qquad \beta_2 = K + 2 - \frac{\alpha_2}2.
\]
If \(\alpha_1\) is odd and \(\alpha_2\) is even then the \(\beta_j\) are integers. If instead \(\alpha_1\) is even and \(\alpha_2\) is odd we apply the refined \(L^\infty\) and \(L^2\) estimates of Remark~\ref{rem:OddSobolev} (using $k \geq 1$, and since $4 \leq \alpha_1 \leq 4+2(K-k)$ and $5 \leq \alpha_2 \leq 4 + 2(K-k)$) to obtain the desired bound.
\end{proof}

\begin{rem}\label{rem:Compact2}
From Remark~\ref{rem:Compact1} and the proof of Lemma~\ref{lem:Multilinear} we see that whenever \(f^{(j)}\rightharpoonup f\) in \(\X^{4,K}\) we may pass to a subsequence to ensure that
\[\<y\>^kL_n[y,\overbrace{f^{(j)},\dots,f^{(j)}}^M]\rightharpoonup \<y\>^kL_n[y,\overbrace{f,\dots,f}^M]
\]
in \(L^2\) for any \(0\leq k\leq K\) and \(0\leq n\leq 2(K - k) + 4\).
\end{rem}

\subsection{Linear estimates}
We complete this section by considering a priori estimates for a model linear equation,
\begin{equation}\label{ModelSmooth}
\begin{cases}
\bw_t + (1 + \bg)\bw_{yyy} + \beta \bg_y\bw_{yy} + \ba \bw_y + \bff = - \nu\bw_{yyyy}\vspace{.1cm}\\
\bw(0) = 0,
\end{cases}
\end{equation}
where \(\beta\in\R\), \(\nu\geq0\) are constants and \(\bg,\ba,\bff\) are sufficiently smooth functions. This will provide a model for the equation satisfied by \(\partial_y^n W\) and will subsequently by used to obtain uniform (in \(\nu\)) bounds for solutions.

Our main a priori estimate for solutions to \eqref{ModelSmooth} is the following:
\begin{prop}\label{prop:MainAP}
Let \(T>0\) and suppose that
\begin{equation}\label{BasicAPriori}
\|\bg\|_{L^\infty([0,T];L^\infty)}\leq\frac12,\qquad \|\bg\|_{L^\infty([0,T];W^{3,\infty})} + \|\ba\|_{L^\infty([0,T];W^{1,\infty})}\lesssim 1.
\end{equation}
Then, if \(\bw\) is a sufficiently smooth, localized solution to \eqref{ModelSmooth} and \(0\leq \nu\ll1\) is sufficiently small we have the estimate
\begin{equation}\label{ModelAPriori}
\|\bw\|_{L^\infty([0,T];L^2)}^2 + \nu \|\bw_{yy}\|_{L^2((0,T);L^2)}^2\lesssim \sigma(T) \|\bff\|_{L^\infty([0,T];L^2)}^2,
\end{equation}
where
\begin{equation}\label{mu}
\sigma(T) = e^{C\int_0^T \left(1 + \|\bg_t\|_{L^\infty}\right)\,dt} - 1,
\end{equation}
and the constants are independent of \(\nu\).

Further, we have the weighted estimate for \(k\geq 1\),
\begin{equation}\label{ModelAPrioriWeighted}
\begin{aligned}
&\|\<y\>^k\bw\|_{L^\infty([0,T];L^2)}^2 + \nu \|\<y\>^k\bw_{yy}\|_{L^2((0,T);L^2)}^2\\
&\quad\lesssim \sigma(T)\left(\|\<y\>^k\bff\|_{L^\infty([0,T];L^2)}^2 + \sum\limits_{n = 0}^2\|\<y\>^{k - 1}\partial_y^n\bw\|_{L^\infty([0,T];L^2)}^2\right),
\end{aligned}
\end{equation}
where again the constants are independent of \(\nu\).
\end{prop}
\begin{proof}
Differentiating with respect to time and integrating by parts we obtain
\begin{align*}
\frac d{dt} \|(1 + \bg)^{\frac13\beta - \frac12} \bw\|_{L^2}^2 &= 2\< (1 + \bg)^{\frac23\beta - 1}\bw,\bw_t\> + (\frac23\beta - 1) \<(1 + \bg)^{\frac23\beta - 2}\bg_t\bw,\bw\>\\
&= - \frac13\beta\<((1 + \bg)^{\frac23\beta - 1}\bg_y)_{yy}\bw,\bw\>  + \<((1 + \bg)^{\frac23\beta - 1}\ba)_y\bw,\bw\>\\
&\quad - 2\<(1 + \bg)^{\frac23\beta - 1}\bff,\bw\> + (\frac23\beta - 1)\<(1 + \bg)^{\frac23\beta - 2}\bg_t\bw,\bw\>\\
&\quad - 2\nu\|(1 + \bg)^{\frac13\beta - \frac12}\bw_{yy}\|_{L^2}^2 + 2\nu\<((1 + \bg)^{\frac23\beta - 1})_{yy}\bw_y,\bw_y\>\\
&\quad - 2\nu\<((1 + \bg)^{\frac23\beta - 1})_{yy}\bw,\bw_{yy}\>.
\end{align*}
We note that from the hypothesis \eqref{BasicAPriori} we have \(1 + \bg \sim 1\). As a consequence we may interpolate to obtain,
\[
\|\bw_y\|_{L^2}^2\lesssim\|(1 + \bg)^{\frac13\beta - \frac12}\bw\|_{L^2}\|(1 + \bg)^{\frac13\beta - \frac12}\bw_{yy}\|_{L^2}.
\]
Choosing \(0<\nu\ll1\) sufficiently small we may apply the hypothesis \eqref{BasicAPriori} to obtain the estimate,
\[
\frac d{dt} \|(1 + \bg)^{\frac13\beta - \frac12} \bw\|_{L^2}^2 \lesssim \left(1 + \|\bg_t\|_{L^\infty}\right) \|(1 + \bg)^{\frac13\beta - \frac12} \bw\|_{L^2}^2 + \|(1 + \bg)^{ \frac13\beta -\frac12}\bff\|_{L^2}^2 - \nu \|(1 + \bg)^{\frac13\beta - \frac12}\bw_{yy}\|_{L^2}^2.
\]
The estimate \eqref{ModelAPriori} then follows from Gronwall's inequality.

To prove \eqref{ModelAPrioriWeighted} we define \(\tilde\bw = y^k \bw\) and \(\mc L = (1 + \bg)\partial_y^3 + \beta\bg_y\partial_y^2 + \ba\partial_y + \nu\partial_y^4\). We then observe that that \(\tilde\bw\) satisfies the equation \eqref{ModelSmooth} with \(\bff\) replaced by
\[
\tilde\bff = y^k\bff - [\mc L,y^k]\bw
\]
Integrating by parts in the terms involving \(\nu\) we obtain the estimate
\[
|\< (1 + \bg)^{\frac23\beta - 1}\tilde\bff,\tilde\bw\>| \lesssim \left(\|y^k\bff\|_{L^2} + \sum\limits_{n = 0}^2\|\<y\>^{k - 1}\partial_y^n\bw\|_{L^2}\right)\|(1 + \bg)^{\frac13\beta - \frac12}\tilde \bw\|_{L^2} + \nu \|\<y\>^k\bw_y\|_{L^2}^2.
\]
In order to bound \( \nu \|\<y\>^k\bw_y\|_{L^2}^2\), we wish to replace the term \(- 2\nu\|(1 + \bg)^{\frac13\beta - \frac12}\tilde\bw_{yy}\|_{L^2}^2\) that appears in the expression for \(\frac d{dt}\|(1 + \bg)^{\frac13\beta - \frac12}\tilde\bw\|_{L^2}^2\) by the term \(- 2\nu\|(1 + \bg)^{\frac13\beta - \frac12}y^k\bw_{yy}\|_{L^2}\). Consequently, we integrate by parts to obtain
\[
\left|\|(1 + \bg)^{\frac13\beta - \frac12}\tilde\bw_{yy}\|_{L^2}^2 - \|(1 + \bg)^{\frac13\beta - \frac12}y^k\bw_{yy}\|_{L^2}^2\right| \lesssim \|\<y\>^k\bw_y\|_{L^2}^2 + \|\<y\>^k\bw\|_{L^2}^2,
\]
and by interpolation we have,
\[
\|\<y\>^k\bw_y\|_{L^2}^2\lesssim \|(1 + \bg)^{\frac13\beta - \frac12}\<y\>^k\bw\|_{L^2}\|(1 + \bg)^{\frac13\beta - \frac12}\<y\>^k\bw_{yy}\|_{L^2} + \|(1 + \bg)^{\frac13\beta - \frac12}\<y\>^k\bw\|_{L^2}^2.
\]
Proceeding as in the proof of \eqref{ModelAPriori} we obtain the estimate \eqref{ModelAPrioriWeighted} whenever \(0\leq \nu\ll1\) is sufficiently small.
\end{proof}


\section{Local well-posedness for \(W\)}\label{sect:Main}

\subsection{Reformulating the problem}
\subsubsection{Lagrangian coordinates}
Considering the hydrodynamic form of  \eqref{QdV}
$$
u_t + (bu)_x = 0, \quad \mbox{where} \quad b = \frac{1}{2} (u^2)_{xx} + \mu u^2,
$$
and recalling the definition \eqref{VFFlow} of the Lagrangian map \(X(t,x)\),
\[
\begin{cases}
X_t(t,x) = b(t,X(t,x))\vspace{0.1cm}\\
X(0,x) = x,
\end{cases}
\]
we may write sufficiently smooth solutions to \eqref{QdV} in the form
\begin{equation}\label{LagrangianForm}
u(t,X(t,x)) = \frac1{X_x(t,x)}u_0(x).
\end{equation}

Assuming that the map \(X\) is sufficiently smooth, we define
\[
Z(t,x) = \frac1{X_x(t,x)} - 1,
\]
and compute the equation satisfied by \(Z\),
\begin{equation}\label{Zeqn}
Z_t + (1 + Z)^2\left(\frac12(1 + Z)\left((1 + Z)\left((1 + Z)^2\rho\right)_x\right)_x + \mu (1 + Z)^2\rho\right)_x = 0,
\end{equation}
where \(\rho = u_0^2\) is defined as above. We note that the equation \eqref{Zeqn} is an inhomogeneous equation with forcing term
\begin{equation}\label{Inhomogeneous}
F := \left( \frac12\rho_{xx} + \mu \rho\right)_x.
\end{equation}
In particular, \(F = 0\) for all \(x\in I\) whenever \(\rho\) corresponds to the initial data for a traveling wave solution of \eqref{QdV}.

\subsubsection{Change of independent coordinates $x \to y$}
The leading order linear part of \eqref{Zeqn} is given by,
\[
Z_t + \rho Z_{xxx} + \text{lower order terms} = 0.
\]
This motivates a change of variables, defining
\begin{equation}\label{COV}
y(x) = \int_0^x \frac1{\rho(s)^{\frac13}}\,ds,
\end{equation}
so that the map \(y\colon I\rightarrow \R\) is a diffeomorphism.

Next we compute the equation \eqref{Zeqn} in these coordinates,
\begin{equation}\label{Flattened}
Z_t + (1 + Z)^5Z_{yyy} + \frac52\frac{\rho_y}\rho(1 + Z)^5Z_{yy} + 7(1 + Z)^4Z_yZ_{yy} + R(y,Z,Z_y) + F = 0,
\end{equation}
where \(R\) is a polynomial in \(Z,Z_y\) satisfying \(R(y,0,0) = 0\) (see \eqref{Rdefn} for the explicit expression) and in the new coordinates the inhomogeneous term \eqref{Inhomogeneous} becomes
\begin{equation}\label{Inhomogeneousy}
F = \frac12\left(\frac{\rho_{yyy}}{\rho} - \frac43\frac{\rho_{yy}\rho_y}{\rho^2} + \frac 59\frac{\rho_y^3}{\rho^3}\right) + \mu \frac{\rho_y}{\rho^{\frac13}}.
\end{equation}
For completeness, the full computation is given in \eqref{FlattenedFull}. We remark that here and subsequently we slightly abuse notation writing \(Z(t,y)\) instead of \(Z(t,x(y))\) and similarly for \(\rho,F\).

\subsubsection{Change of dependent coordinates $Z \to W$}
In order to work in unweighted \(L^2\) spaces we take \(W = \rho^{\frac56}Z\). The equation \eqref{Flattened} may then be written as
\begin{equation}\label{FinalForm}
W_t + (1 + g)W_{yyy} + \frac75 g_yW_{yy} + N(y,W,W_y) + \rho^{\frac 56}F = 0,
\end{equation}
where \(N\) is a polynomial in \(W,W_y\) satisfying \(N(y,0,0) = 0\) (see \eqref{Ndefn} for the explicit expression) and we define
\begin{equation}\label{Metric}
g = (1 + \rho^{-\frac 56}W)^5 - 1.
\end{equation}
We will then consider the existence of solutions to \eqref{FinalForm} in the weighted Sobolev spaces \(\X^{N,K}\) defined as in \eqref{FunctionSpace}.

\subsection{The initial data}
We now describe our assumptions on the initial data \(\rho = u_0^2\), which are most easily stated in the \(y\)-coordinates. However, they may be phrased in the original coordinates using the change of variables \eqref{COV} and we compute a couple of special cases in Section~\ref{sect:Examples}.

We first make the assumption that there exists an integer \(K_0\geq 0\) and some \(\delta>0\) so that in the \(y\)-coordinates,
\begin{equation}\label{SlowDecayInit}
\inf\limits_{y\in \R} \rho(y)^{\frac56}\<y\>^{K_0} \geq \delta.
\end{equation}
It seems reasonable to expect this hypothesis is true whenever (in the \(x\)-coordiantes) \(\rho\in C^3(\R)\) has subcritical decay at both endpoints in the sense of Definition~\ref{df:Subcrit}. We verify that it is indeed true for polynomially decaying data in Section~\ref{sect:Examples}.

Next we assume that
\begin{equation}\label{InitBound}
\|\rho\|_{L^\infty} + \sum\limits_{n = 1}^{2K_0 + 7}\left\|\frac{\partial_y^n\rho}\rho\right\|_{L^2}\lesssim 1.
\end{equation}
Finally we assume that there exists some \(\cM>0\) so that the inhomogeneous term \(F\), defined as in \eqref{Inhomogeneousy}, satisfies,
\begin{equation}\label{L2Init}
\|\rho^{\frac56}F\|_{\X^{4,K_0}}\leq \cM,
\end{equation}
where the integer \(K_0\geq0\) appears in the lower bound \eqref{SlowDecayInit}.

\begin{rem}
For most estimates we will treat \(\rho\) and its derivatives as coefficients in the linear and nonlinear terms involving \(W\). In this case it will be more convenient to use that from the estimate \eqref{InitBound} and Sobolev embedding we have the pointwise bound,
\begin{equation}\label{PointwiseInit}
\| \rho \|_{L^\infty} + \sum\limits_{n = 1}^{2K_0 + 6}\left\|\frac{\partial_y^n \rho}\rho\right\|_{L^\infty}\lesssim1.
\end{equation}
The only exception to this will be when \(2K_0 + 7\) derivatives fall on \(\rho\), where we will instead use the estimate \eqref{InitBound} directly.
\end{rem}

\begin{rem}
We note when \(\mu = 1\) the assumptions on \(\rho\) do not preclude the case that \(\rho \in (\Phi_{B,c})^2 + \epsilon C^\infty_0(\R)\) where \(c>0\), \(-\frac14 c^2\leq B<0\) and \(0<\epsilon\ll_{B,c}1\), i.e. \(\rho\) is a small perturbation of the non-degenerate traveling wave \(\Phi_{B,c}\).
\end{rem}

\begin{rem}
The assumptions \eqref{L2Init},\eqref{PointwiseInit} are far from the optimal regularity. In future work we will show that it is possible to improve the regularity by taking further advantage of the dispersive smoothing effects similarly to~\cite{MR3263550,MR2955206,MR3417686,MR3280026}.
\end{rem}

The main result of this section is the existence of solutions to the equation \eqref{FinalForm}:
\begin{thm}\label{thm:ProxyLWP}
Suppose that in the \(x\)-coordinates \(\rho = u_0^2 \in C^3(\R)\) satisfies the subcritical left and right endpoint decay conditions \eqref{LEP} and \eqref{REP}. Suppose also that in the \(y\)-coordinates \(\rho\) satisfies the estimates \eqref{SlowDecayInit}, \eqref{InitBound} and \eqref{L2Init}. Then there exists a time \(T>0\) and a unique (classical) solution \(W\in C([0,T];\X^{4,K_0})\) to the equation \eqref{FinalForm}.
\end{thm}

\begin{rem}
We note that using the usual frequency envelope approach it is possible to show that in the \(y\)-coordinates the map \(\ln \rho\mapsto W\) is continuous as map from \(L^\infty\cap\dot H^1\cap \dot H^{2K_0+7}\rightarrow\X^{4,K_0}\). However, as the y-coordinate is defined in terms of \(\rho\), this does not imply continuous dependence on the initial data for the original equation \eqref{QdV}. Similarly, the uniqueness stated in Theorem~\ref{thm:ProxyLWP} does not imply uniqueness for \eqref{QdV} so we must apply Theorem~\ref{thm:Unique} instead. As a consequence we omit the proof of continuity of the solution map for \eqref{FinalForm} and only include the proof of uniqueness because it is brief.
\end{rem}

\subsection{Two particular cases}\label{sect:Examples}

In order to better understand the conditions~\eqref{SlowDecayInit}--\eqref{PointwiseInit}, we will illustrate them in two specific cases.

\bigskip

\noindent \underline{Case 1: $\operatorname{Supp} \rho = [-1,1]$, with $\rho(s) \sim (1-s)^\alpha$, $\alpha>3$, as $s \to 1$}. By this, we mean that $\rho$ is sufficiently smooth in $(-1,1)$, and that for sufficiently many derivatives of $\rho$, there holds $\partial^k_s \rho(s) = C_k (1-s)^{\alpha-k} + o((1-s)^{\alpha-k})$, for a constant $C_k \in \mathbb{R}$. Notice that we only discuss here the right endpoint, but the left endpoint can of course be dealt with symmetrically. Then
$$
y(x) = \int_0^x \frac{ds}{(1-s)^{\alpha/3}} \sim (1-x)^{1-\frac{\alpha}{3}}.
$$
This implies that, in the coordinate $y$,
$$
\rho(y) = \rho(x(y)) \sim y^{\frac{3\alpha}{3-\alpha}} \quad \mbox{and} \quad F(y) \sim y^{-3}.
$$
Therefore, the condition~\eqref{InitBound} always holds, while the conditions~\eqref{SlowDecayInit} and~\eqref{L2Init} become, respectively,
$$
K_0 > \frac{5}{2} \frac{\alpha}{\alpha-3} \quad \mbox{and} \quad K_0 < \frac{5}{2} \frac{2\alpha -3}{\alpha - 3}.
$$
For $\alpha > 3$, there exists an integer $K_0$ satisfying these two constraints.

\bigskip

\noindent \underline{Case 2: $\operatorname{Supp} \rho = \mathbb{R}$, with $\rho(s) \sim s^{-\beta}$, $\beta \geq 0$, as $s \to \infty$} (once again, the case $s\to - \infty$ can be dealt with symmetrically). Then
$$
y(x) = \int_0^x s^{\beta/3}\,ds \sim x^{1 + \frac{\beta}{3}}.
$$
This implies that
$$
\rho(y) \sim y^{- \frac{3\beta}{3+\beta}} \quad \mbox{and} \quad F(y) \sim y^{-\frac{3\beta + 3}{\beta + 3}}.
$$
Therefore, the condition~\eqref{InitBound} always holds, while the conditions~\eqref{SlowDecayInit} and~\eqref{L2Init} become, respectively,
$$
K_0 > \frac{5}{2} \frac{\beta}{\beta + 3} \quad \mbox{and} \quad K_0 < \frac{10 \beta +3}{2\beta + 6}.
$$
For $\beta \geq 0$, there exists an integer $K_0$ satisfying these two constraints.

\subsection{Existence of solutions}
We now consider a parabolic regularization of \eqref{FinalForm} with initial data \(W_0\),
\begin{equation}\label{FinalFormReg}
\begin{cases}
W_t + (1 + g)W_{yyy} + \frac75g_{y}W_{yy} + N + \rho^{\frac56}F = - \nu \partial_y^4W,\vspace{0.1cm}\\
W(0) = W_0,
\end{cases}
\end{equation}
where \(g\) is defined as in \eqref{Metric}, \(N\) as in \eqref{Ndefn} and \(F\) as in \eqref{Inhomogeneousy}.

We then have the following existence result:
\begin{lem}\label{lem:MildExistence}
Let $\delta > 0$ be the constant defined in \eqref{SlowDecayInit} and \(W_0\in \X^{4,K_0}\) satisfy the estimate
\begin{equation}\label{ContiInit}
\|W_0\|_{\X^{4,K_0}}\leq \delta.
\end{equation}
Then, for each \(0<\nu\ll 1\) sufficiently small there exists a time \(T = T(\nu)>0\) and a (mild) solution \(W\in\cZ^{4,K_0}\) of the equation \eqref{FinalFormReg}.
\end{lem}

\begin{proof}
We take \(B\subset \cZ^{4,K_0}\) to be the ball
\[
B = \left\{ W\in \cZ^{4,K_0}: \|W\|_{\cZ^{4,K_0}}\leq\cK \delta\right\},
\]
where the constant \(\cK\sim 1\) may be chosen independently of \(\rho,W_0,\nu\). We then define
\[
\mc T[W] = e^{-\nu t\partial_x^4}W_0 + \int_0^t e^{-\nu(t - s)\partial_y^4}\left(G(s) + \rho^{\frac56}F\right)\,ds,
\]
where
\[
G = (1 + g)W_{yyy} + \frac75g_yW_{yy} + N.
\]

From the semigroup estimate \eqref{semigroupInit} and the estimate \eqref{ContiInit} we have the estimate
\[
\|e^{-\nu\partial_x^4}W_0\|_{\cZ^{4,K_0}}\lesssim \|W_0\|_{\X^{4,K_0}}\lesssim \delta.
\]
Similarly, from the estimate \eqref{L2Init} for the inhomogeneous term \(F\) and the semigroup estimate \eqref{ParabolicReg} we have the estimate,
\[
\left\|\int_0^t e^{-\nu(t - s)\partial_y^4} \rho^{\frac56}F\,dy\right\|_{\cZ^{4,K_0}}\lesssim \nu^{-\frac34}T^{\frac14}\cM,
\]
provided \(T\lesssim \nu^{-1}\).

For \(0\leq n\leq 2K_0 + 3\) we may write \(\partial_y^nG\) as a multilinear operator of the form,
\[
\partial_y^nG = \sum\limits_{M = 1}^6\sum\limits_{|\alpha|\leq 3}L_n[y,\overbrace{\partial_y^{\alpha_1}W,\dots,\partial_y^{\alpha_M}W}^M],
\]
where the coefficients of the \(L_n\) may be uniformly bounded in \(L^\infty\) using the lower bound \eqref{SlowDecayInit} and pointwise estimate \eqref{PointwiseInit} for \(\rho\). When \(n = 2K_0 + 4\) we may instead write
\[
\begin{aligned}
\partial_y^{2K_0 + 4}G &= \sum\limits_{M = 1}^6\sum\limits_{|\alpha|\leq 3}L_{2K_0 + 4}[y,\overbrace{\partial_y^{\alpha_1}W,\dots,\partial_y^{\alpha_M}W}^M]\\
&\quad + \frac{\partial_y^{2K_0 + 7}\rho}{\rho}\left(\frac12\rho^{\frac 56}\left((1 + \rho^{-\frac56}W)^5 - 1\right) - \frac13(1 + \rho^{-\frac56} W)^5W\right),
\end{aligned}
\]
where the coefficients of the \(L_{2K_0 + 4}\) are uniformly bounded and the final term may be bounded by estimating \(\dfrac{\partial_y^{2K_0 + 7}\rho}{\rho}\in L^2\) using \eqref{InitBound} and the remaining terms in \(L^\infty\) using \eqref{Sobolev}. As a consequence, we may apply the multilinear estimate \eqref{MultilinearBound} to obtain
\[
\|G\|_{\X^{4,K_0}}\lesssim (\nu t)^{-\frac34} (1 + \delta^{-1}\|W\|_{\cZ^{4,K_0}})^5\|W\|_{\cZ^{4,K_0}}\lesssim (\nu t)^{-\frac34}\delta
\]
provided \(0< T \lesssim \nu^{-1}\). We then apply the semigroup estimate \eqref{ParabolicReg} to obtain
\[
\left\| \int_0^te^{-\nu(t - s)\partial_y^4} G(s)\,ds\right\|_{\cZ^{4,K_0}}\lesssim \nu^{-\frac34} T^{\frac 14}\delta.
\]

Applying identical estimates for the difference \(\mc T[W^{(1)}] - \mc T[W^{(2)}]\) we see that we may choose the timescale \(0<T = T(\nu)\ll 1\) sufficiently small so that the map \(\mc T\colon B\rightarrow B\) is a contraction on \(B\). The result then follows from an application of the contraction principle.
\end{proof}

In order to pass to a limit as \(\nu\rightarrow0\) in the equation \eqref{FinalFormReg} we must prove uniform (in \(\nu\)) estimates for the solutions to \eqref{FinalFormReg}. However, these will follow directly from the a priori estimates for the model equation:
\begin{prop}\label{prop:UniBounds}
Let \(W_0 = 0\) and \(W\in \cZ^{4,K_0}\) be a mild solution of \eqref{FinalFormReg}. Then there exists a time \(T_* = T_*(K_0,\cM,\delta)>0\) so that provided \(0<T\leq T_*\) we have the estimate,
\begin{equation}\label{UniformAPriori}
\|W\|_{L^\infty([0,T];\X^{4,K_0})}^2 + \nu \|W\|_{L^2((0,T);\X^{6,K_0})}^2 \ll \delta^2,
\end{equation}
where the constants are independent of sufficiently small \(0<\nu\ll1\).
\end{prop}
\begin{proof}
We make the bootstrap assumption that for some \(\cK>0\) we have
\begin{equation}\label{BS}
\|W\|_{L^\infty([0,T];\X^{4,K_0})}^2 + \nu \|W\|_{L^2((0,T);\X^{6,K_0})}^2\leq (\cK \delta)^2.
\end{equation}
We then observe that \(W^{(n)} = \partial_y^nW\) satisfies the equation
\begin{equation}\label{P0}
W_t^{(n)} + (1 + g)W_{yyy}^{(n)} + (\frac 75 + n)g_{y}W_{yy}^{(n)} + a^{(n)}W_y^{(n)} + N^{(n)} + \partial_y^n(\rho^{\frac56}F) = -\nu\partial_y^4W^{(n)},
\end{equation}
where the coefficient
\[
a^{(n)} = a^{(n)}(y,W,W_y,W_{yy}),
\]
is a polynomial in \(\rho^{-\frac56}W,\rho^{-\frac56}W_y,\rho^{-\frac56}W_{yy}\) with bounded coefficients, and the perturbative term
\[
N^{(n)} = \sum\limits_{M = 1}^6L_n[y,\overbrace{W,\dots,W}^M],
\]
whenever \(0\leq n\leq 2K_0 + 3\), with the slight modification when \(n = 2K_0 + 4\),
\[
\begin{aligned}
N^{(2K_0 + 4)} &= \sum\limits_{M = 1}^6L_{2K_0 + 7}[y,\overbrace{W,\dots,W}^M]\\
&\quad + \frac{\partial_y^{2K_0 + 7}\rho}{\rho}\left(\frac12\rho^{\frac 56}\left((1 + \rho^{-\frac56}W)^5 - 1\right) - \frac13(1 + \rho^{-\frac56} W)^5W\right)
\end{aligned}
\]
We then note that \eqref{P0} is in the form of the model equation \eqref{ModelSmooth} with \(\bg = g\), \(\beta = \frac75 + n\), \(\ba = a^{(n)}\), \(\bff = N^{(n)} + \partial_y^n(\rho^{\frac56}F)\).

Applying the Sobolev estimate \eqref{Sobolev} and the pointwise estimate \eqref{PointwiseInit} for \(\rho\) with the bootstrap assumption \eqref{BS} we may bound
\begin{equation}\label{P1}
\|g\|_{W^{3,\infty}}\lesssim (1 + \cK)^4\cK,\qquad \|a^{(n)}\|_{W^{1,\infty}}\lesssim (1 + \cK)^5.
\end{equation}
In particular, the coefficients satisfy the hypothesis \eqref{BasicAPriori} of Proposition~\ref{prop:MainAP} whenever \(0<\cK\ll1\) is sufficiently small.

From the Sobolev estimate \eqref{Sobolev} we may bound,
\[
\left\|g_t\right\|_{L^\infty}\lesssim \left(1 + \cK\right)^4\delta^{-1}\|W_t\|_{\X^{1,K_0}}.
\]
We then use the equation \eqref{FinalFormReg} to write,
\[
W_t + G + \rho^{\frac56}F = - \nu W_{yyyy},
\]
where, for \(0\leq n\leq 2K_0 + 1\),
\[
\partial_y^nG = \sum\limits_{M = 1}^6L_{n +3}[y,\overbrace{W,\dots,W}^M].
\]
Applying the multilinear estimate \eqref{MultilinearBound} with the pointwise estimate \eqref{PointwiseInit} for \(\rho\) and \(\X^{4,K_0}\)-estimate \eqref{L2Init} for \(F\) we then obtain
\[
\|W_t\|_{\X^{1,K_0}}\lesssim \left(1 + \delta^{-1}\|W\|_{\X^{4,K_0}}\right)^5\|W\|_{\X^{4,K_0}} + \cM + \nu\|W\|_{\X^{5,K_0}}.
\]
As a consequence we may use the bootstrap assumption \eqref{BS} to obtain
\begin{equation}\label{P3}
\int_0^T\|g_t\|_{L^\infty}\,dt\lesssim (1 + \cK)^4\left((1 + \cK)^5\cK T + \delta^{-1}\cM T + \cK\sqrt{\nu T}\right).
\end{equation}

Finally, we apply the multilinear estimate \eqref{MultilinearBound} (and the Sobolev estimate \eqref{Sobolev} for the final term when \(n = 2K_0 + 4\)) to obtain,
\begin{equation}\label{P3_1}
\|\<y\>^k N^{(n)} \|_{L^2}\lesssim(1 + \cK)^5 \cK \delta ,
\end{equation}
whenever \(0\leq k\leq K_0\) and \(0\leq n\leq 2(K_0 - k) + 4\).

Choosing \(0<\cK\ll1\) and \(0<\nu \ll1\) sufficiently small we may then apply Proposition~\ref{prop:MainAP} (noting that it applies to mild solutions via a standard approximation argument) to obtain the estimate
\[
\|W\|_{L^\infty([0,T];\X^{4,K_0})}^2 + \nu \|W\|_{L^2((0,T);\X^{6,K_0})}^2 \lesssim \sigma(T)\mc M^2,
\]
where
\begin{equation}
\label{mudef}
\sigma(T) = e^{C(1 + \delta^{-1}\cM)T + C\sqrt{\nu T}} - 1,
\end{equation}
and the constant is independent of \(\nu\). Note that we use repeatedly~\eqref{ModelAPrioriWeighted}, along with the elementary inequality $(e^{C_1}-1)(e^{C_2}-1) \leq e^{C_1+C_2}-1$ for $C_1, C_2 > 0$ to absorb these terms into the constant $C=C(K_0)$.

We may thus find a $T_*$ independent of $\nu$ such that for all $0 < T \leq T_*$ we have
\[
\|W\|_{L^\infty([0,T];\X^{4,K_0})}^2 + \nu \|W\|_{L^2((0,T);\X^{6,K_0})}^2 \leq \frac12 (\cK \delta)^2,
\]
allowing us to close the bootstrapping argument for existence.
\end{proof}

\subsection{Uniqueness of solutions}
We now consider the linearization of \eqref{FinalForm}, taking \(\bw  \) and \(\brho\), \(\mbf F \) to be the first variations of \(W,\rho, F\) respectively to obtain the equation
\begin{equation}\label{Linearization}
\begin{cases}
\bw_t + (1 + g)\bw_{yyy} + \frac75g_y\bw_{yy} + \ba \bw_y + \mbf b \bw + \bff + \rho^{\frac56}\mbf F = 0,\vspace{0.1cm}\\
\bw(0) = 0,
\end{cases}
\end{equation}
where we define \(g\) as in \eqref{Metric}, the coefficients \(\ba = \ba(\rho,W), \mbf b = \mbf b(\rho,W)\) may be bounded using the Sobolev estimate \eqref{Sobolev} and the estimates \eqref{SlowDecayInit}, \eqref{InitBound} for \(\rho\) so that for each \(0\leq k\leq K_0\) we have
\begin{equation}\label{Linearized1}
\sum\limits_{n = 0}^{2(K_0 - k) + 2}\|\<y\>^k\partial_y^n\ba\|_{L^\infty} + \sum\limits_{n = 0}^{2(K_0 - k) + 1}\|\<y\>^k\partial_y^n\mbf b\|_{L^\infty}\lesssim (1 + \delta^{-1}\|W\|_{\X^{4,K_0}})^5,
\end{equation}
and the inhomogeneous term \(\bff = \bff(\rho,W,F,\brho)\) may be bounded similarly to obtain
\begin{equation}\label{Linearized2}
\|\bff\|_{\X^{1,K_0}}\lesssim (1 + \delta^{-1}\|W\|_{\X^{4,K_0}})^6\left(\|\brho\|_{L^\infty} + \sum\limits_{n = 1}^{2K_0 + 4}\|\partial_y^n(\ln \pmb\rho)\|_{L^2}\right).
\end{equation}

We then have the following estimate for the linearized equation:
\begin{prop}\label{prop:LinearizedAP}
If \(W\in C([0,T];\X^{4,K_0})\cap C^1([0,T];\X^{1,K_0})\) and \(\bw \in C([0,T];\X^{4,K_0})\) is a solution of the equation \eqref{Linearization} then we have the estimate
\begin{equation}\label{LinearizedAPriori}
\|\bw\|_{L^\infty([0,T];\X^{1,K_0})}\lesssim \sigma(T) \left(\|\brho\|_{L^\infty} + \sum\limits_{n = 1}^{2K_0 + 4}\|\partial_y^n(\ln \pmb\rho)\|_{L^2} + \|\rho^{\frac56}\mbf F\|_{\X^{1,K_0}}\right),
\end{equation}
where \(\sigma(T) = O(T)\) as \(T\rightarrow 0\) and the constants depend on \(\|W\|_{L^\infty([0,T];\X^{4,K_0})},\|W_t\|_{L^\infty([0,T];\X^{1,K_0})}\).
\end{prop}
\begin{proof}
Proceeding as in the proof of Proposition~\ref{prop:UniBounds} we apply the a priori estimate for the model equation \eqref{ModelSmooth} with \(\nu = 0\) with the estimates \eqref{Linearized1}, \eqref{Linearized2} for the coefficients and inhomogeneous term. The details are left to the reader.
\end{proof}

\begin{cor}\label{cor:Uniqueness}
Solutions to \eqref{FinalForm} are unique in the space \(C([0,T];\X^{4,K_0})\cap C^1([0,T];\X^{1,K_0})\).
\end{cor}
\begin{proof}
Given any two solutions \(W^{(1)},W^{(2)}\) of \eqref{FinalForm} we define
\[
W^{(\tau)} = \tau W^{(1)} + (1 - \tau)W^{(2)}.
\]
We then see that the difference \(\bw = W^{(1)} - W^{(2)}\) satisfies the linearized equation \eqref{Linearization} about \(W^{(\tau)}\) with \(\brho = 0\) and \(\bff = 0\) integrated from \(\tau = 0\) to \(\tau  = 1\). Observing that the proof of the estimate \eqref{LinearizedAPriori} may be applied with \(\bg\) replaced by \(\int_0^1\bg\,d\tau\), etc., we may proceed as in Proposition~\ref{prop:LinearizedAP} to show that \(\bw = 0\).
\end{proof}

\subsection{Proof of Theorem~\ref{thm:ProxyLWP}}
We now complete the proof of Theorem~\ref{thm:ProxyLWP}. The argument is an essentially standard application of the energy method, so we only sketch the details:

\medskip
\begin{enumerate}
\item \textbf{Existence of solutions to the regularized equation \eqref{FinalFormReg}.} We first apply Lemma~\ref{lem:MildExistence} for each \(0<\nu\ll1\) sufficiently small to obtain a solution \(W^{(\nu)}\in \cZ^{4,K_0}\) of the regularized equation \eqref{FinalFormReg} with \(W_0 = 0\).
\item \textbf{Uniform bounds.} Next we apply Proposition~\ref{prop:UniBounds} with Lemma~\ref{lem:MildExistence} and a standard bootstrap argument to find a time \(T>0\) independent of \(\nu\) so that the set \(\{W^{(\nu)}\}\) is uniformly bounded in \(C([0,T];\X^{4,K_0})\). Further, as \(W^{(\nu)}\) is a mild solution of \eqref{FinalFormReg} we see that \(\{W_t^{(\nu)}\}\) is uniformly bounded in \(L^2([0,T];\X^{1,K_0})\).
\item \textbf{Existence of a solution to the equation \eqref{FinalForm}.} By weak compactness there exists a weak limit point \(W\in L^\infty([0,T];\X^{4,K_0})\cap H^1((0,T);\X^{1,K_0})\) satisfying the estimate
\begin{equation}\label{LimitPoint}
\|W\|_{L^\infty([0,T];\X^{4,K_0})}\ll \delta.
\end{equation}
Further, from the compactness of the Sobolev embedding \eqref{Sobolev} (see Remarks~\ref{rem:Compact1},~\ref{rem:Compact2}), by passing to a subsequence \(\nu_j\rightarrow0\) we may take a limit in \eqref{FinalFormReg} to show that \(W\) is a distributional solution of \eqref{FinalForm}. In particular, \(W\in L^\infty([0,T];\X^{4,K_0})\cap W^{1,\infty}((0,T);\X^{1,K_0})\) satisfies the equation \eqref{FinalForm} almost everywhere.
\item\textbf{Continuity in time.} It remains to show that the solution \(W\in C([0,T];\X^{4,K_0})\). To do this we define the mollified data,
\[
\ln\rho_{\leq j} = P_{\leq j}(\ln \rho),\qquad (\rho^{\frac56}F)_{\leq j} = P_{\leq j}(\rho^{\frac56}F).
\]
From the estimate \eqref{InitBound} for \(\rho\) we see that
\[
|\ln \rho - \ln\rho_{\leq j}|\lesssim 2^{-j}\|\partial_y\ln \rho\|_{L^\infty}\lesssim 2^{-j},
\]
and hence by shrinking \(\delta\) slightly we may ensure that \(\rho_{\leq j}\) satisfies the lower bound \eqref{SlowDecayInit} whenever \(j\gg1\). We note that \(\rho_{\leq j}\) satisfies the \(L^2\)-estimate \eqref{InitBound} and \((\rho^{\frac56}F)_{\leq j}\) satisfies the estimate \eqref{L2Init} uniformly in \(j\gg1\). Further, we have the estimates,
\begin{equation}\label{MollyHigh}
\|\ln \rho_{\leq j}\|_{\dot H^{2K_0 + 4 + n}}\lesssim 2^{nj},\qquad \|(\rho^{\frac56} F)_{\leq j}\|_{\X^{4 + n,K_0}} \lesssim 2^{nj}\cM,
\end{equation}
whenever \(n\geq 0\). Finally, we note that as \(j\rightarrow\infty\) we have,
\begin{equation}\label{MollyDiff}
\|\ln \rho_{\leq j} - \ln \rho\|_{\dot H^1\cap \dot H^{2K_0 + 1}} =o( 2^{-3j} ),\qquad \|(\rho^{\frac56}F)_{\leq j} - (\rho^{\frac56}F)\|_{\X^{1,K_0}} = o(2^{-3j}).
\end{equation}

Repeating the proof of the existence of \(W\), after shrinking the time \(T\) slightly we may find a solution \(W_{\leq j}\in L^\infty([0,T];\X^{4,K_0})\cap W^{1,\infty}((0,T);\X^{1,K_0})\) to the equation \eqref{FinalForm} with \(\rho\) replaced by \(\rho_{\leq j}\) and \(\rho^{\frac56}F\) by \((\rho^{\frac 56}F)_{\leq j}\). However, differentiating the equation we obtain an equation that is still of the form of the model equation and hence we may apply essentially identical estimates to Lemma~\ref{lem:MildExistence} and Proposition~\ref{prop:UniBounds} to show that
\[
W_{\leq j}\in L^\infty([0,T];\X^{7,K_0})\cap W^{1,\infty}((0,T);\X^{4,K_0}),
\]
satisfies the estimate
\begin{equation}\label{HB}
\|W_{\leq j}\|_{L^\infty([0,T];\X^{7,K_0})}\lesssim 2^{3j}\delta.
\end{equation}
By redefining \(W_{\leq j}\) on a set of measure zero we may also assume that \(W_{\leq j}\in C([0,T];\X^{4,K_0})\).

Next we consider the equation for the difference \(W - W_{\leq j}\). Estimating the difference \(W - W_{\leq j}\) using the a priori estimate for the linearized equation \eqref{LinearizedAPriori} as in the proof of Corollary~\ref{cor:Uniqueness}, and applying the estimate \eqref{MollyDiff} we then obtain
\begin{equation}\label{LB}
\|W - W_{\leq j}\|_{L^\infty([0,T];\X^{1,K_0})} = o( 2^{-3j}\delta ),\qquad j\rightarrow\infty.
\end{equation}
Applying the interpolation estimate \eqref{Interpolate} with the estimate \eqref{HB} we may then show that the sequence \(W_{\leq j}\) is Cauchy in \(C([0,T];\X^{4,K_0})\). Further, from \eqref{LB} the limit is given by \(W\) and hence \(W\in C([0,T];\X^{4,K_0})\).
\item \textbf{Uniqueness.} This follows from Corollary~\ref{cor:Uniqueness}.
\end{enumerate}
\medskip
This completes the proof of Theorem~\ref{thm:ProxyLWP}.\qed

\section{Existence and uniqueness of solutions to \eqref{QdV}}
In this section we prove a rigorous version of Theorem~\ref{thm:MainRough}, giving the existence and uniqueness of hydrodynamic solutions to \eqref{QdV} for the set of initial data considered in Theorem~\ref{thm:ProxyLWP}.

\begin{thm}\label{thm:MainFull}
Suppose that in the \(x\)-coordinates \(\rho = u_0^2 \in C^3(\R)\) satisfies the subcritical left and right endpoint decay conditions \eqref{LEP} and \eqref{REP}. Suppose also that in the \(y\)-coordinates \(\rho\) satisfies the estimates \eqref{SlowDecayInit}, \eqref{InitBound} and \eqref{L2Init}. Then there exists a time \(T>0\) and a unique hydrodynamic solution to the equation \eqref{QdV}.
\end{thm}

To prove Theorem~\ref{thm:MainFull} we first reverse the derivation of the equation \eqref{FinalForm} and apply Theorem~\ref{thm:ProxyLWP} to construct a solution. We then prove Theorem~\ref{thm:Unique} to show that this is the unique hydrodynamic solution of the problem.

\subsection{Existence}
Given initial data as in Theorem~\ref{thm:MainFull}, we may apply Theorem~\ref{thm:ProxyLWP} to obtain a solution \(W\in C([0,T];\X^{4,K_0})\cap C^1([0,T];\X^{1,K_0})\) satisfying the equation \eqref{FinalForm}. Taking \(Z = \rho^{-\frac56}W\) we may use the lower bound \eqref{SlowDecayInit} to show that \(Z\in C([0,T];H^4)\cap C^1([0,T];H^1)\) is a classical solution of the equation \eqref{Flattened}.

Next we invert the change of coodinates \eqref{COV} and extend \(Z\) to \(\R\) by zero to obtain a solution of the equation \eqref{Zeqn} on \(\R\), where we note that, by applying Sobolev embedding in the \(y\)-coordinates, in the \(x\)-coordinates we have \(\rho^{\frac n3}\partial_x^nZ\in C_b([0,T]\times \R)\) for \(n = 0,1,2,3\).

Na\"ively we wish to define the Lagrangian map \(X\) by taking \(X_x = \dfrac1{1 + Z}\). However, this only defines \(X\) up to a time-dependent constant. To choose the constant we define
\[
U(t,x) = \left(1 + Z(t,x)\right)u_0(x),
\]
and observing that \(\rho^{\frac n3-1}\partial_x^n(U^2)\in C_b([0,T]\times\R))\) we may define
\[
B = \frac12(1 + Z)\left((1+Z)(U^2)_x\right)_x + \mu U^2 \in C([0,T];C^1_b(\R)).
\]
Using this, we find the characteristic passing through \((t,x) = (0,0)\) by finding a solution \(\xi\in C^1([0,T])\) of the ODE
\[
\begin{cases}
\dot\xi(t) = B(t,0),\vspace{.1cm}\\
\xi(0) = 0.
\end{cases}
\]
We may then define
\[
X(t,x) = \xi(t) + \int_0^x\frac1{1 + Z(t,s)}\,ds,
\]
where we note that from the proof of Theorem~\ref{thm:ProxyLWP} we have
\[
\sup\limits_{t\in[0,T]}\|Z\|_{L^\infty}\ll1,
\]
and hence \(X\in C^1([0,T]\times\R)\). By construction, it satisfies
$$
X_t(t,x) = B(t,x) \qquad \mbox{for $(t,x) \in [0,T] \times \mathbb{R}$}.
$$

The map \(x\mapsto X(t,x)\) is a diffeomorphism so we may find an inverse \(Y\in C^1([0,T]\times \R)\) so that \(Y_x(t,x) = (1 + Z(t,Y(t,x)))\) and hence \(\rho(Y(t,x))^{\frac n3}\partial_x^{n+1}Y(t,x)\in C_b([0,T]\times\R)\) for \(n = 1,2,3\). We then define
\[
u(t,x) = Y_x(t,x) u_0(Y(t,x)),
\]
and observe that \(u\in C^1_b([0,T]\times \R)\) and \(u^2\in C([0,T];C^3_b(\R))\). Further, with this definition we see that
\[
B(t,x) = b(t,X(t,x)),
\]
where \(b = \frac12(u^2)_{xx} + \mu u^2\). In particular, \(X\) satisfies the ODE \eqref{VFFlow} (recalling that $Z(0,x) = 0$) and hence \(u\) is a hydrodynamic solution of \eqref{QdV}. Further, using the bounds on \(Y\) it is straightforward to verify that \(u^{\frac43}\in C([0,T];C^2_b(\R))\) and hence satisfies the hypothesis of Theorem~\ref{thm:Unique}.

\subsection{Uniqueness}\label{sect:Uniqueness}~
We now prove Theorem~\ref{thm:Unique}: the uniqueness of hydrodynamic solutions.

We first note that if \(u\) is a hydrodynamic solution of \eqref{QdV} then \(w = u^2\) is a non-negative classical solution of the equation
\begin{equation}\label{QdV2}
\begin{cases}
w_t + 2\left(\frac12w_{xx} + \mu w\right)_x w + \left(\frac12w_{xx} + \mu w\right)w_x= 0,\vspace{0.1cm}\\
w(0) = w_0 := u_0^2.
\end{cases}
\end{equation}

Next we define the Lagrangian map \(X\) as in \eqref{VFFlow} and, treating \(b\) as a fixed function, uniqueness of solutions to linear transport equations ensures that \(u\) may be written in the form,
\[
u(t,x) = Y_x(t,x) u_0(Y(t,x)),
\]
where \(Y\in C^1([0,T]\times \R)\) is the inverse of the map \(x\mapsto X(t,x)\). From the ODE satisfied by \(Y_x\) we obtain the estimate
\[
|\ln Y_x|\lesssim \int_0^t \|u^2\|_{W^{3,\infty}}\,ds,
\]
so as \(u_0\geq0\) we have \(u\geq 0\). In particular, provided classical solutions to \eqref{QdV2} are unique, so are hydrodynamic solutions to \eqref{QdV}.

Taking \(\bw\) the first variation \(w\) in the equation \eqref{QdV2}, we have
\begin{equation}\label{LinearizedQdV2}
\begin{cases}
\bw_t + w\bw_{xxx} + \frac12 w_x\bw_{xx} + (\frac12w_{xx} + 3\mu w)\bw_x + (w_{xxx} + 3\mu w_x)\bw = 0,\vspace{0.1cm}\\
\bw(0) = \bw_0.
\end{cases}
\end{equation}
We then have the following lemma:
\begin{lem}
Suppose that \(w\in C([0,T];C^3_b(\R))\cap C^1([0,T];C_b(\R))\) is a non-negative classical solution of \eqref{QdV2} such that \(w^{\frac23}\in C([0,T];C^2_b(\R))\), and \(\bw\in C([0,T];C^3_b(\R))\cap C^1([0,T];C_b(\R))\) is a classical solution of \eqref{LinearizedQdV2} with \(w_0^{-\frac13}\bw_0\in L^2(\R)\). Then \(w^{-\frac13}\bw\in C([0,T];L^2(\R))\) satisfies the estimate
\begin{equation}\label{QdV2LinearizedBound}
\|w^{-\frac13}\bw\|_{L^\infty([0,T];L^2)}\lesssim e^{CT}\|w_0^{-\frac13}\bw_0\|_{L^2}.
\end{equation}
\end{lem}
\begin{proof}
Replacing \(w\) by \(\sqrt{\epsilon^2 + w^2}\) and then taking a limit as \(\epsilon\rightarrow0\) it suffices to assume that \(w>0\). Further, by a standard approximation argument we may assume that \(\bw\) has compact support. Integrating by parts we then obtain
\begin{align*}
\frac d{dt}\|w^{-\frac13}\bw\|_{L^2}^2 &= -\frac23\<w^{-\frac53}w_t\bw,\bw\> - 2\<w^{\frac13}\bw,\bw_{xxx}\> - \<w^{-\frac23}w_x\bw,\bw_{xx}\> - \<w^{-\frac23}w_{xx}\bw,\bw_x\> \\
&\quad - 2\<w^{-\frac23}w_{xxx}\bw,\bw\> - 6\mu\<w^{\frac13}\bw,\bw_x\> - 6\mu\<w^{-\frac23}w_x\bw,\bw\>\\
&= -\frac23\<w^{-\frac23}w_{xxx}\bw,\bw\> - \frac49\<w^{-\frac53}w_xw_{xx}\bw,\bw\> + \frac 5{27}\<w^{-\frac83}w_x^3\bw,\bw\> - 5\mu\<w^{-\frac23}w_x\bw,\bw\>.
\end{align*}
As \(w(t)^{\frac23}\in C^2_b(\R)\) is non-negative, a simple argument of Glaeser~\cite{MR0163995} shows that \(w(t)^{\frac13}\in W^{1,\infty}\) and we may bound
\[
\|w^{-\frac23}w_x\|_{L^\infty} + \|w^{-\frac13}w_{xx}\|_{L^\infty} + \|w_{xxx}\|_{L^\infty}\lesssim \|w^{\frac23}\|_{W^{2,\infty}} + \|w\|_{W^{3,\infty}}.
\]
As a consequence, using the equation \eqref{QdV2} to bound \(\dfrac{w_t}w\) we have the estimate
\[
\frac d{dt}\|w^{-\frac13}\bw\|_{L^2}^2\lesssim \|w^{-\frac13}\bw\|_{L^2}^2,
\]
and the estimate \eqref{QdV2LinearizedBound} then follows from Gronwall's inequality.
\end{proof}

Arguing as in the proof of Corollary~\ref{cor:Uniqueness} we may then use the estimate \eqref{QdV2LinearizedBound} to show that any two hydrodynamic solutions \(u,\tilde u\)  satisfying the hypothesis of Theorem~\ref{thm:Unique} with initial data \(u_0,\tilde u_0\) satisfy the estimate
\[
\|u^{\frac43} - \tilde u^{\frac43}\|_{L^2}\lesssim_{u,\tilde u,T} \|u_0^{\frac43} - \tilde u_0^{\frac 43}\|_{L^2},
\]
and hence solutions are unique.

This completes the the proof of Theorem~\ref{thm:Unique} and hence of Theorem~\ref{thm:MainFull}.\qed

\section{The virial argument}\label{sect:Virial}

For the convenience of the reader, and since it is short and elegant, we recall here the virial argument of Zilburg and Rosenau in \cite{2017arXiv170903322Z} in the focusing case $\mu =1$; we further observe that an analogous approach works in the defocusing case $\mu = -1$ and that this approach applies to hydrodynamic solutions, defined as in Definition~\ref{def:hydro}.

\subsection{Hydrodynamic solutions}

We will be dealing with solutions $u \in C^1 ([0,T] \times \mathbb{R})$, \(u^2\in C([0,T];C^3(\R))\) satisfying the hydrodynamic formulation of \eqref{QdV},
$$
u_t + (bu)_x = 0, \qquad \mbox{where} \qquad b = \frac12(u^2)_{xx} + (u^2).
$$
It is clear that these solutions conserve the Hamiltonian $H$, mass $M$, and momentum $J$ and that these solutions propagate non-negativity or non-positivity: $u \geq 0$ or $u \leq 0$.

\subsection{The focusing case $\mu = 1$}

\begin{lem}
\label{plover1}
Assume that $u_0 \neq 0$ and $H(u_0) \geq 0$. Then there does not exist a globally defined hydrodynamic solution $u \in C^1 ([0,\infty) \times \mathbb{R})$, \(u^2\in C([0,\infty);C^3(\R))\), and a real number $M > 0$ such that
\begin{equation}
\label{annashummingbird}
\forall t \in \mathbb{R}, \qquad \supp u(t,\cdot) \subset [-M,M].
\end{equation}
\end{lem}
\begin{proof} Pairing~\eqref{QdV} with $xu$ and integrating by parts leads to the identity
$$
\frac{1}{2} \frac{d}{dt} \int x u^2 \,dx = - \frac{5}{2} \int |uu_x|^2\,dx - \frac{3}{4} \int |u|^4\,dx.
$$
Using the conservation of the Hamiltonian, this becomes
$$
\frac{1}{2} \frac{d}{dt} \int x u^2 \,dx = - \int |uu_x|^2\,dx - 3 H(u_0).
$$
Combined with the conservation of the mass $M$, this identity leads to a contradiction under the hypothesis~\eqref{annashummingbird}. \end{proof}

\subsection{The defocusing case $\mu = -1$}

\begin{lem}
\label{plover2}
Assume that $u_0 \geq 0$ or $u_0 \leq 0$. Then there does not exist a globally defined hydrodynamic solution $u \in C^1 ([0,\infty) \times \mathbb{R})$, \(u^2\in C([0,\infty);C^3(\R))\), and a real number $M > 0$ such that
\begin{equation}
\label{annashummingbird2}
\forall t \in \mathbb{R}, \qquad \supp u(t,\cdot) \subset [-M,M].
\end{equation}
\end{lem}
\begin{proof}
Multiplying\eqref{QdV} by $x$ and integrating leads to the identity
$$
\frac{d}{dt} \int x u \,dx = - \int u |u_x|^2\,dx + \mu \int u^3 \,dx.
$$
Combined with the conservation of the momentum $J$, this identity leads to a contradiction under the hypothesis~\eqref{annashummingbird2}.
\end{proof}

\subsection{The neutral case $\mu = 0$}

In the neutral case, both Lemma~\ref{plover1} and Lemma~\ref{plover2} hold true.

\begin{appendix}

\section{Coordinate changes}
\subsection{The equation for $Z$ in $x$ coordinates}

Let us start with the hydrodynamic formulation
$$
u_t + (bu)_x = 0, \qquad \mbox{where} \qquad b = \frac12(u^2)_{xx} + \mu u^2.
$$
Denote $X(t,x)$ the Lagrangian map defined by
$$
\left\{
\begin{array}{l}
X_t(t,x) = b(t,X(t,x)) \\
X(t=0,x) = x.
\end{array}
\right.
$$
Differentiating the above in $x$ and letting $Z(t,x) = \frac{1}{X_x(t,x)} - 1$ leads to
$$
Z_t(t,x) = - (Z(t,x)+1)b_x(t,X(t,x)).
$$
We observe that $u(t,X(t,x)) = (Z(t,x) +1) u_0(x)$ and hence, denoting $Y(t,\cdot)$ the inverse map to $X(t,\cdot)$,
$$
u^2(t,x) = \left[ (Z(t,\cdot) + 1)^2 u_0^2 \right] (Y(t,x)).
$$
We further observe that, for any function $F$,
$$
\partial_x [ F(t,Y(t,x))] = \left[ (1 + Z(t,\cdot)) \partial_y F(t,\cdot) \right] (Y(t,x)).
$$
Introducing the notation $\rho = (u_0)^2$, this leads to
$$
Z_t + (1 + Z)^2\left(\frac12(1 + Z)\left((1 + Z)\left((1 + Z)^2\rho\right)_x\right)_x + \mu (1 + Z)^2\rho\right)_x = 0,
$$
which becomes after expanding
\begin{equation}\label{ZeqnFull}
\begin{aligned}
&Z_t + \rho(1 + Z)^5Z_{xxx} + \frac72\rho_x(1 + Z)^5Z_{xx} + 7\rho(1 + Z)^4Z_xZ_{xx}\\
&\quad + \frac92\rho_{xx}(1 + Z)^5Z_x + \frac{19}2\rho_x(1 + Z)^4Z_x^2 + 4\rho(1 + Z)^3Z_x^3\\
&\quad + \frac12\rho_{xxx}(1 + Z)^6 + \mu \rho_x(1 + Z)^4 + 2\mu \rho(1 + Z)^3Z_x = 0.
\end{aligned}
\end{equation}

\subsection{The equation for $Z$ in $y$ coordinates}
Making the change of variables 
$$
y(x) = \int_0^x \frac{ds}{\rho(s)^{1/3}},
$$
we have
\begin{gather*}
\partial_x = \rho^{-\frac13}\partial_y,\qquad \partial_x^2 = \rho^{-\frac23}\left(\partial_y^2 - \frac13\frac{\rho_y}\rho\partial_y\right),\\
\partial_x^3 = \rho^{-1}\left(\partial_y^3 - \frac{\rho_y}\rho\partial_y^2 + \frac59\frac{\rho_y^2}{\rho^2} \partial_y - \frac13\frac{\rho_{yy}}\rho\partial_y\right),
\end{gather*}
and therefore the equation becomes
\begin{equation}\label{FlattenedFull}
\begin{aligned}
&Z_t + (1 + Z)^5Z_{yyy} + \frac52\frac{\rho_y}{\rho}(1 + Z)^5 Z_{yy} + 7(1 + Z)^4Z_yZ_{yy}\\
&\quad - \frac{19}9\frac{\rho_y^2}{\rho^2}(1 + Z)^5Z_y + \frac{25}6\frac{\rho_{yy}}{\rho}(1 + Z)^5Z_y + \frac{43}6\frac{\rho_y}{\rho}(1 + Z)^4Z_y^2 + 4(1 + Z)^3Z_y^3\\
&\quad + \frac12\left(\frac{\rho_{yyy}}{\rho} - \frac43\frac{\rho_y\rho_{yy}}{\rho^2} + \frac 59\frac{\rho_y^3}{\rho^3}\right)(1 + Z)^6 + \mu \rho^{\frac23}\left(\frac{\rho_y}{\rho}(1 + Z)^4  + 2(1 + Z)^3Z_y\right) = 0.
\end{aligned}
\end{equation}
As a consequence, we obtain the equation
\begin{equation*}
Z_t + (1 + Z)^5Z_{yyy} + \frac52\frac{\rho_y}\rho(1 + Z)^5Z_{yy} + 7(1 + Z)^4Z_yZ_{yy} + R(y,Z,Z_y) + F = 0,
\end{equation*}
where
$$
F = \frac12\left(\frac{\rho_{yyy}}{\rho} - \frac43\frac{\rho_{yy}\rho_y}{\rho^2} + \frac 59\frac{\rho_y^3}{\rho^3}\right) + \mu \frac{\rho_y}{\rho^{\frac13}}.
$$
and
\begin{equation}\label{Rdefn}
\begin{aligned}
R(y,Z,Z_y) &= - \frac{19}9\frac{\rho_y^2}{\rho^2}(1 + Z)^5Z_y + \frac{25}6\frac{\rho_{yy}}{\rho}(1 + Z)^5Z_y + \frac{43}6\frac{\rho_y}{\rho}(1 + Z)^4Z_y^2 + 4(1 + Z)^3Z_y^3\\
&\quad + \frac12\left(\frac{\rho_{yyy}}{\rho} - \frac43\frac{\rho_y\rho_{yy}}{\rho^2} + \frac 59\frac{\rho_y^3}{\rho^3}\right)\left((1 + Z)^6 - 1\right)\\
&\quad + \mu \rho^{\frac23}\left(\frac{\rho_y}{\rho}\left((1 + Z)^4 - 1\right)  + 2(1 + Z)^3Z_y\right).
\end{aligned}
\end{equation}

\subsection{The equation for $W$}
Writing \(Z = \rho^{-\frac56}W\), we obtain the expressions
\begin{align*}
Z_y &= \rho^{-\frac56}\left(W_y - \frac 56\frac{\rho_y}\rho W\right),\\
Z_{yy} &= \rho^{-\frac56}\left(W_{yy} - \frac53\frac{\rho_y}\rho W_y - \frac 56\frac{\rho_{yy}}\rho W + \frac{55}{36}\frac{\rho_y^2}{\rho^2}W\right),\\
Z_{yyy} &= \rho^{-\frac56}\left(W_{yyy} - \frac52 \frac{\rho_y}\rho W_{yy} - \frac52\frac{\rho_{yy}}\rho W_y + \frac{55}{12}\frac{\rho_y^2}{\rho^2}W_y - \frac 56\frac{\rho_{yyy}}\rho W + \frac{15}{12}\frac{\rho_y\rho_{yy}}{\rho^2}W - \frac{935}{216}\frac{\rho_y^3}{\rho^3}W\right).
\end{align*}
Plugging these into \eqref{FlattenedFull} leads to
\begin{equation*}
W_t + (1 + g)W_{yyy} + \frac75 g_yW_{yy} + N(y,W,W_y) + \rho^{\frac 56}F = 0,
\end{equation*}
where we define
$$
g = (1 + \rho^{-5/6} W)^5 - 1
$$
and
\begin{equation}\label{Ndefn}
\begin{aligned}
N(y,W,W_y) &=  - \frac52\frac{\rho_{yy}}\rho(1 + \rho^{-\frac56}W)^5 W_y + \frac5{12}\frac{\rho_y^2}{\rho^2}(1 + \rho^{-\frac56}W)^5W_y - \frac 56\frac{\rho_{yyy}}\rho (1 + \rho^{-\frac56}W)^5W\\
&\quad - \frac{55}{108}\frac{\rho_y^3}{\rho^3}(1 + \rho^{-\frac56}W)^5W - \frac {5}{6}\frac{\rho_y\rho_{yy}}{\rho^2}(1 + \rho^{-\frac56}W)^5 W\\
&\quad + 7(1 + \rho^{-\frac56}W)^4(\rho^{-\frac56}W)_y\left( - \frac53\frac{\rho_y}\rho W_y - \frac 56\frac{\rho_{yy}}\rho W + \frac{55}{36}\frac{\rho_y^2}{\rho^2}W\right)\\
&\quad + \rho^{\frac56}R\left(y,\rho^{-\frac56}W,(\rho^{-\frac56}W)_y\right),
\end{aligned}
\end{equation}
and finally \(R\) is defined as in \eqref{Rdefn}.
\end{appendix}


\bibliographystyle{abbrv}
\bibliography{KdV}
\bigskip
\end{document}